\documentclass[a4paper]{article}
\usepackage[round]{natbib}
\usepackage{amsmath, amsfonts}
\usepackage{amsthm}
\usepackage{subcaption}
\usepackage{tikz}
\usepackage{chapterbib}
\usepackage{titling}

\newtheorem{theorem}{Theorem}
\newtheorem{lemma}{Lemma}

\newcommand{\norm}[1]{\lVert#1\rVert}

\title{On asymptotic validity of naive inference with an approximate likelihood}
\author{Helen Ogden}
\date{University of Southampton, UK}

\begin{document}

\maketitle
\subsection*{Abstract}

Many statistical models have likelihoods which are intractable:
 it is impossible or too expensive to compute the likelihood 
exactly.
In such settings, a common approach is to replace
the likelihood with an approximation, 
and proceed with inference
as if the approximate likelihood were the exact likelihood.
In this paper, we describe conditions on the approximate likelihood
which guarantee that this naive inference with an approximate likelihood
has the same first-order asymptotic properties as
inference with the exact likelihood.
We investigate the implications of these results
for inference using a Laplace approximation to the likelihood
in a simple two-level latent variable model, and using
reduced dependence approximations to the likelihood in an
Ising model on a lattice.

\textbf{Keywords}: Intractable likelihood,
Ising model,
Laplace approximation,
Latent variable model

\section{Introduction}

For many models, it is impossible
or infeasibly expensive to evaluate the likelihood function,
typically because it involves a high-dimensional sum or integral.
In such cases, a common approach is to find an
approximation $\tilde L(.)$ to the likelihood $L(.)$, and to use 
$\tilde L(.)$
in place of $L(.)$ to conduct inference about the parameters 
of the model.

For instance, one could construct a point estimate of the parameters
by maximizing the approximate likelihood, and form
confidence intervals based on the curvature of the
approximated log-likelihood about its maximum.
From a Bayesian perspective,  an
approximate posterior
$\tilde \pi(\theta; y) \propto \tilde L(\theta; y) \pi(\theta)$
could be formed
by substituting the approximate likelihood in place of the exact
likelihood.

Such an approach is commonly used in practice.
In latent variable models, where the likelihood
is an integral over the latent variables, naive inference
using a Laplace approximation to the likelihood
is used in both maximum likelihood \citep{Pinheiro1995, Bates2015}
and Bayesian \citep{Rue2009} settings.
In Markov random field models,
where the likelihood involves
an intractable normalizing constant, inference
is often conducted
by substituting an approximation to this normalizing
constant in place of the exact normalizing constant
into the expression for the likelihood \citep{Friel2009, Tjelmeland2012}.

In this paper, we provide conditions under which
the naive approach is asymptotically justified. 
Under these conditions, the
approximate maximum likelihood estimator
is consistent and has the same asymptotic normal
distribution as the exact maximum likelihood estimator,
hypothesis tests based on the approximate likelihood
remain valid, and in Bayesian analysis the distance between
the approximate posterior and the true posterior
shrinks to zero.

\cite{Douc2004} show that the
approximate maximum likelihood estimator
will have the correct
asymptotic normal distribution provided that
the error in the log-likelihood,
\[\epsilon_n(\theta) = \log \tilde L_n(\theta) - \log L_n(\theta),\]
tends in probability to zero as $n \rightarrow \infty$,
uniformly in $\theta$.
We argue that this measure is too strict in many practical
examples,
 in which $\epsilon_n(\theta)$
grows rapidly with $n$ and yet the inference remains
asymptotically valid.
Our conditions are based instead 
on $\nabla_\theta \epsilon_n(\theta)$,
the error in the approximation to the score function.

We provide two examples to demonstrate
how the conditions may be used in practice.
The first is a simple two-level latent variable
model, with $m_n$ repeated observations for
each of $n$ items. We
deduce the rate at which $m_n$
must grow with $n$ in order for the Laplace approximation
to give asymptotically valid inference. If $m_n$
grows with $n$ at a slower rate, the estimator remains
consistent, but loses efficiency relative to the exact maximum likelihood estimator,
and naively constructed confidence intervals have lower than
nominal coverage.

The second example is
an Ising model on an $m \times m$ lattice, with
the class of reduced dependence approximations
\citep{Friel2009} used to approximate the likelihood.
For parameter values associated
with weak dependence, we show that
the reduced dependence approximation may be used to
obtain asymptotically valid inference at cost polynomial
in $m$, in contrast with the exponential cost of
computing the likelihood exactly.

\section{Asymptotic validity of approximate likelihood
  inference}
\label{sec:results_likelihood}

\subsection{Setup and notation}

Consider a sequence of models indexed by $n$,
with common parameter $\theta \in \Theta \subseteq{\mathbb{R}^p}$. Write
$\ell_n(\theta; y)$ for the log-likelihood given observed data
$y$ under model $n$,
$u_n(\theta; y) = \nabla_\theta \ell_n(\theta; y)$ for the corresponding
score function and
$J_n(\theta; y) = -\nabla_\theta^T \nabla_\theta \ell_n(\theta; y)$
for the observed information. We sometimes drop the data
$y$ from the notation for convenience.
Suppose that the data were generated from the
model for some $\theta_0 \in \Theta$, and that
as $n \rightarrow \infty$,
the amount of information provided by the data about the parameter
grows at some rate $r_n$, such that 
$J_n(\theta_0) = O(r_n)$ in probability.
This includes the case with $n$
independent replications as a special case, with $r_n = n$,
but we also wish to allow for more complex settings.
Write $I(\theta)$ for the
Fisher information matrix, chosen such that
$\bar J_n(\theta) = r_n^{-1} J_n(\theta) \rightarrow I(\theta)$
in probability.

Let $\tilde \ell_n(.; y)$ be an approximate log-likelihood,
which in general may be any
function of the parameters $\theta$, which
will be used in place of $\ell_n(.; y) \null$. 
Write $\hat \theta_n$ and $\tilde \theta_n$ for the estimators
maximizing $\ell_n(\theta)$ and $\tilde \ell_n(\theta)$
respectively.
Suppose that $\ell_n(\theta)$
and $\tilde \ell_n(\theta)$ are both three times
differentiable, and write
$\tilde u_n(\theta) = \nabla_\theta \tilde \ell_n(\theta)$
and
$\tilde J_n(\theta) = -\nabla_\theta^T \nabla_\theta \tilde \ell_n(\theta)$
for the approximate score and information.

Write 
\[\epsilon_n(\theta) = \tilde \ell_n(\theta) - \ell_n(\theta)\]
for the pointwise error in the log-likelihood,
\[\delta_n(\theta) = \norm{\nabla_\theta \epsilon_n(\theta)}
= \norm{\tilde u_n(\theta) - u_n(\theta)} \]
for the absolute error in the score,
and
\[\gamma_n(\theta) = \norm{\nabla_\theta^T \nabla_\theta \epsilon_n(\theta)} = \norm{J_n(\theta) - \tilde J_n(\theta)}\]
for the absolute error in the observed information matrix.
For concreteness, we use the $L_1$ norms
$\norm{a} = \sum_{i} |a_i|$ for a vector $a$, and
$\norm{A} = \max_{j} \{\sum_i |A_{ij}|\}$ for a matrix $A$,
although the same results hold for any choice of norms.

Write
$\delta_n^\infty (S) = \sup_{\theta \in S} \delta_n(\theta)$
for the uniform error in the score over any set $S \subseteq \Theta$,
and let $\delta_n^\infty = \delta_n^\infty(\Theta)$.
Similarly, define $\gamma_n^\infty(S) = \sup_{\theta \in S} \gamma_n(\theta)$ and $\gamma_n^\infty = \gamma_n^\infty(\Theta)$.
For any $\theta_0 \in \Theta$, write
$B_t(\theta_0) = \{\theta \in \Theta: \norm{\theta - \theta_0} \leq t\}$
for the ball of radius $t$ about $\theta_0$.

\subsection{Approximate maximum likelihood inference}
First, we describe sufficient conditions to ensure that $\tilde \theta_n$
is consistent. The proofs of all results
are given in the appendix.

We will assume
some standard regularity conditions on the model.
Writing
$\bar u_n(\theta) = r_n^{-1} u_n(\theta)$, and
$\bar u(\theta)$ for the limit as $n \rightarrow \infty$,
we assume that
$\sup_{\theta \in \Theta} \norm{\bar u_n(\theta) - \bar u(\theta)} \rightarrow 0$
in probability.
We assume $\bar u(.)$ is such that, for any $\epsilon > 0$,
$\int_{\theta: d(\theta, \theta_0) \geq \epsilon} \norm{\bar u(\theta)} > \bar u(\theta_0) = 0$.
These conditions are stronger than necessary, and
we expect the same result to hold
in many other situations where the exact maximum likelihood
estimator is consistent.

\begin{theorem}
  \label{thm:consistency}
  Suppose $\delta_n^\infty = o_p(r_n)$ as $n \rightarrow \infty \null$.
  Then $\tilde \theta_n \rightarrow \theta_0$
  in probability, as $n \rightarrow \infty$.
\end{theorem}

We now give conditions to ensure that $\tilde \theta_n$
retains the same limiting distribution as $\hat \theta_n$.
Since $\hat \theta_n - \theta_0$
is $O_p(r_n^{-1/2})$, this is equivalent to finding
conditions under which $\hat \theta_n - \tilde \theta_n$
is $o_p(r_n^{-1/2})$. The following lemma bounds the
distance between $\tilde \theta_n$ and $\hat \theta_n$
in terms of the error in the score function near $\theta_0$.
\begin{lemma}
\label{lemma:distance_estimators}
Suppose that $\delta_n^\infty = o_p(r_n)$,
and that there exists $t > 0$ such that
$\delta_n^\infty\{B_t(\theta_0)\} = o_p(a_n) \null$. Then
$\tilde \theta_n - \hat \theta_n = o_p(a_n r_n^{-1})$.
\end{lemma}

Applying Lemma \ref{lemma:distance_estimators} 
 with $a_n = r_n^{1/2}$
 leads directly to the asymptotic normality result.

\begin{theorem}
  \label{thm:normality}
Suppose that $\delta_n^\infty = o_p(r_n)$, and that 
there exists $t > 0$ such that
$\delta_n^\infty\{B_t(\theta_0)\} = o_p(r_n^{1/2})$.
  Then
  \[r_n^{1/2} (\tilde \theta_n - \theta_0) \rightarrow N(0, I(\theta_0)^{-1})\]
 in distribution, as $n \rightarrow \infty$.
\end{theorem}

It is also desirable for
hypothesis tests constructed by using the approximate likelihood
in place of the exact likelihood
to have the correct asymptotic distribution.

Consider testing the hypothesis $H_0: \theta \in \Theta_R$,
where $\Theta_R \subset \Theta$ and $\text{dim}(\Theta_R) = q$.
Write $\hat \theta_n^R$ for the restricted maximum likelihood estimator, and
$\Lambda_n = 2 \{\ell_n(\hat \theta_n) - \ell_n(\hat \theta_n^R) \}$
for the likelihood ratio statistic.
The approximate version of the likelihood ratio test statistic
is
$\tilde \Lambda_n = 2 \{\tilde \ell_n(\tilde \theta_n) - \tilde \ell_n(\tilde \theta_n^R)\}$,
where $\tilde \theta_n^R$ is the restricted
approximate likelihood estimator.

Under the same conditions that were used to show that
$\tilde \theta_n$ has the correct limiting distribution,
plus a bound on the error in the information around
$\theta_0$, $\tilde \Lambda_n$ is asymptotically
equivalent to $\Lambda_n$ under $H_0$.

\begin{theorem}
  \label{thm:hypothesis}
Suppose that $\delta_n^\infty = o_p(r_n)$,
 and that 
there exists $t > 0$ such that
$\delta_n^\infty\{B_t(\theta_0)\} = o_p(r_n^{1/2})$
and $\gamma_n^\infty\{B_t(\theta_0)\} = o_p(r_n) \null$.
Then, under $H_0$, $\tilde \Lambda_n - \Lambda_n = o_p(1)$.
\end{theorem}

The Wald and score test statistics,
$W_n$ and $S_n$, are asymptotically
equivalent to the likelihood ratio test,
so under $H_0$, all three statistics have limiting
distribution $\chi^2_{p - q}$.
Under the conditions of Theorem \ref{thm:hypothesis},
 $I(\theta_0)$ is consistently estimated by
$r_n^{-1} \tilde J_n(\tilde \theta_n)$,
so the approximate  Wald and score
test statistics $\tilde W_n$ and $\tilde S_n$
are also asymptotically equivalent to $\Lambda_n$.

\subsection{Approximate Bayesian inference}

We now consider the approximate posterior
\[\tilde \pi(\theta |y) \propto \tilde L_n(\theta; y) \pi(\theta),\]
where we suppose that the prior is such that $\log \pi(.)$
 is three times differentiable.
Under the same conditions that were used to show asymptotic
correctness of maximum likelihood inference, the total
variation distance between the approximate and exact posteriors,
  \[d_{TV}\{\tilde \pi(\theta|y), \pi(\theta | y)\}
  = \frac{1}{2} \int_{\Theta} \big|\tilde \pi(\theta|y) - \pi(\theta|y) \big| d\theta,\]
  tends to zero as $n \rightarrow \infty$.

  \begin{theorem}
    \label{thm:posterior}
Suppose that $\delta_n^\infty = o_p(r_n)$,
 and that 
there exists $t > 0$ such that
$\delta_n^\infty\{B_t(\theta_0)\} = o_p(r_n^{1/2})$
and $\gamma_n^\infty\{B_t(\theta_0)\} = o_p(r_n) \null$.
  Then
  \[d_{TV}\{\tilde \pi(\theta|y), \pi(\theta | y)\} = o_p(1)\]
  as $n \rightarrow \infty$.
  \end{theorem}

  If the Bernstein-von Mises Theorem holds for the
  exact posterior distribution,
  under the conditions of Theorem  \ref{thm:posterior},
  it will also hold for the  approximate posterior
  distribution $\tilde \pi(\theta | y) \null$. In that case,
  credible regions formed from the approximate posterior
  distribution will also be valid confidence sets.

\subsection{Adjusted approximate likelihood inference}

If Lemma \ref{lemma:distance_estimators} holds for some
$a_n > r_n^{1/2}$, then $\tilde \theta_n$ will still
be a consistent estimator, but
might not have the same limiting distribution as $\hat \theta_n$.
The approximate likelihood may still be useful in practice,
provided that the inference is adjusted accordingly.

The sandwich information matrix \citep{Godambe1960} is
$G_n(\theta) = I_n(\theta) H_n(\theta)^{-1} I_n(\theta)$
where $H_n(\theta) = \text{var} \{ \nabla_\theta \tilde u_n(\theta) \}$
and $I_n(\theta) = E\{\tilde J_n(\theta)\}$.
Under suitable regularity conditions,
\[s_n^{1/2} (\tilde \theta_n - \theta_0) \rightarrow N(0, \bar G(\theta_0)^{-1})\]
in distribution, where $s_n$ 
 is the rate of convergence of
 $\tilde \theta_n$, chosen such that $G_n(\theta_0) = O(s_n)$,
 and $\bar G(\theta) = \lim_{n \rightarrow \infty} s_n^{-1} G_n(\theta)$.

Composite likelihood estimators \citep{Lindsay1988} also have
this type of asymptotic behaviour, and many methods
which have been proposed to adjust inference using a
composite likelihood could also be used to adjust the inference
with an approximate likelihood, provided that $\tilde \theta_n$
is a consistent estimator.
For example, \cite{Varin2011} describe various methods
to approximate the variance of a composite likelihood
estimator, and list some modification to 
the composite likelihood ratio test statistic designed to ensure
that the resulting test statistic has approximately $\chi^2_{p-q}$
distribution. From a Bayesian perspective, the adjustments proposed
by \cite{Pauli2011} and \cite{Ribatet2012}
to posterior distributions based on composite likelihoods
 could also be used in the context of approximate likelihood inference.
  
\section{Examples}
\subsection{Laplace approximation to the likelihood in a simple
latent variable model}
\label{example:two_level}

\subsubsection{A two-level model}

  Suppose $Y_{i} \sim \text{Binomial}(m, p_{i})$, where
  $\text{logit}(p_{i}) = b_i$ and $b_i \sim N(0, \theta^2)$,
  for $i= 1, \ldots, n$.
  The likelihood is
  \begin{align*}
  L_n(\theta) &= \int_{\mathbb{R}^n} \prod_{i=1}^n
  \{\text{logit}^{-1}(b_i)\}^{y_{i}} \{1 - \text{logit}^{-1}(b_i)\}^{m - y_i} \phi(b_i; 0, \theta) db \\
  &=\prod_{i=1}^n L_{(i)}(\theta),
  \end{align*}
  where
 \[L_{(i)}(\theta) =  \int_{-\infty}^\infty \{\text{logit}^{-1}(b_i)\}^{y_{i}} \{1 - \text{logit}^{-1}(b_i)\}^{m_n - y_i} \phi(b_i; 0, \theta) db_i \]
  and
  $\phi(.; \mu, \sigma)$ is the $N(\mu, \sigma^2)$ density
  function.

  If we take a Laplace approximation $\tilde L_n(\theta)$
  to the likelihood,
  it is intuitively clear that $m = m_n$ will have to grow with $n$ to give
  valid inference as $n \rightarrow \infty$. It is less obvious whether
  any choice of $m_n$ that grows with $n$ will give
  valid inference, or whether $m_n$ needs to grow with $n$ at
  some minimum rate.
  \cite{Rue2009} suggest that any $m_n$ which grows with $n$ will suffice,
  conjecturing that
  ``the error rate'' is [number of latent variables / number
    of observations], although they note that this rate
  is not established rigorously. In this case, the error rate refers to the
  error in approximating $\pi(\theta | y)$ with $\tilde \pi(\theta | y)$,
  found by using a Laplace approximation to the likelihood.
  The Integrated Nested Laplace Approximations proposed by
  \cite{Rue2009} are based on this $\tilde \pi(\theta | y)$, with
  further approximations used to approximate the marginal posterior
  distribution of each component of $\theta$,
  if $p > 1 \null$. In this example $\theta$ is a scalar, so the Integrated
  Nested Laplace Approximation to the posterior
  distribution is precisely $\tilde \pi(\theta | y)$.

\subsubsection{Theoretical analysis}
  The factorization of the likelihood allows us to study
  the error for each item
  $\epsilon_{(i)}(\theta) = \ell_{(i)}(\theta) - \ell_{(i)}(\theta)$
  separately, and then combine the errors.
In the supplementary materials, we show that 
for each fixed $\theta \in \Theta$,
  \[\delta_{(i)}(\theta) = \norm{\nabla_\theta \epsilon_{(i)}(\theta)} = O_p(m_n^{-2}),\]
  so
  $\delta_n(\theta) \leq \sum_{i=1}^n \delta_{(i)}(\theta) = O_p(n m_n^{-2})$.
  
  However, the conditions are in terms
  of uniform rather than pointwise errors. Since
   $\delta_{(i)}(\theta)$ is maximized at a point
  $\theta^* = O_p(m_n^{-1/2})$, and decreasing for all
  $\theta > \theta^* \null$,
  \[\delta_{(i)}^\infty = \delta_{(i)}(\theta^*) = O_p(m_n^{-1/2}),\] 
  so $\delta_n^\infty = O_p(n m_n^{-1/2})$.

 The amount of information that the data provides about $\theta$
  is bounded for fixed $n$ as $m_n \rightarrow \infty$,
  by the available information on $\theta$
  given the value of each $b_i \null$.
  So if $m_n \rightarrow \infty$ as $n \rightarrow \infty$, then
  $r_n = O(n)$, and since
  $\delta_n^\infty = o_p(r_n)$,  $\tilde \theta_n$ will
  be consistent.

  Provided that $\theta_0 \not = 0$, choose any
  fixed $t \in (0, \theta_0)$. Then, for sufficiently large
  $n$,
  \[\delta_n^{\infty}\{B_t(\theta_0)\} = \delta_n(\theta_0 - t) = O_p(n m_n^{-2})\]
  and
  \[\gamma_n^\infty\{B_t(\theta_0)\} = \gamma_n(\theta_0 - t) = O_p(n m_n^{-2}).\]
  If $m_n$ grows at a rate faster
  than $n^{1/4}$, then $\delta_n^{\infty}\{B_t(\theta_0)\} = o_p(n^{1/2})$
  and $\gamma_n^{\infty}\{B_t(\theta_0)\} = o_p(n)$, so the
  Laplace approximation to the 
likelihood will give first-order correct inference.

\subsubsection{Numerical demonstration}
  \begin{figure}
    \begin{subfigure}[b]{0.5 \textwidth}
\begin{tikzpicture}[x=1pt,y=1pt]
\definecolor{fillColor}{RGB}{255,255,255}
\path[use as bounding box,fill=fillColor,fill opacity=0.00] (0,0) rectangle (173.45,144.54);
\begin{scope}
\path[clip] ( 24.00, 24.00) rectangle (166.25,138.54);
\definecolor{drawColor}{RGB}{0,0,0}

\path[draw=drawColor,line width= 0.4pt,dash pattern=on 1pt off 3pt ,line join=round,line cap=round] ( 29.27, 28.24) --
	( 43.90, 52.87) --
	( 58.54, 69.84) --
	( 73.17, 83.20) --
	( 87.81, 94.38) --
	(102.44,104.09) --
	(117.08,112.72) --
	(131.71,120.53) --
	(146.34,127.68) --
	(160.98,134.30);
\end{scope}
\begin{scope}
\path[clip] (  0.00,  0.00) rectangle (173.45,144.54);
\definecolor{drawColor}{RGB}{0,0,0}

\path[draw=drawColor,line width= 0.4pt,line join=round,line cap=round] ( 43.90, 24.00) -- (160.98, 24.00);

\path[draw=drawColor,line width= 0.4pt,line join=round,line cap=round] ( 43.90, 24.00) -- ( 43.90, 18.00);

\path[draw=drawColor,line width= 0.4pt,line join=round,line cap=round] ( 73.17, 24.00) -- ( 73.17, 18.00);

\path[draw=drawColor,line width= 0.4pt,line join=round,line cap=round] (102.44, 24.00) -- (102.44, 18.00);

\path[draw=drawColor,line width= 0.4pt,line join=round,line cap=round] (131.71, 24.00) -- (131.71, 18.00);

\path[draw=drawColor,line width= 0.4pt,line join=round,line cap=round] (160.98, 24.00) -- (160.98, 18.00);

\node[text=drawColor,anchor=base,inner sep=0pt, outer sep=0pt, scale=  1.00] at ( 43.90,  8.40) {2000};

\node[text=drawColor,anchor=base,inner sep=0pt, outer sep=0pt, scale=  1.00] at ( 73.17,  8.40) {4000};

\node[text=drawColor,anchor=base,inner sep=0pt, outer sep=0pt, scale=  1.00] at (102.44,  8.40) {6000};

\node[text=drawColor,anchor=base,inner sep=0pt, outer sep=0pt, scale=  1.00] at (131.71,  8.40) {8000};

\path[draw=drawColor,line width= 0.4pt,line join=round,line cap=round] ( 24.00, 28.24) -- ( 24.00,128.86);

\path[draw=drawColor,line width= 0.4pt,line join=round,line cap=round] ( 24.00, 28.24) -- ( 18.00, 28.24);

\path[draw=drawColor,line width= 0.4pt,line join=round,line cap=round] ( 24.00, 45.01) -- ( 18.00, 45.01);

\path[draw=drawColor,line width= 0.4pt,line join=round,line cap=round] ( 24.00, 61.78) -- ( 18.00, 61.78);

\path[draw=drawColor,line width= 0.4pt,line join=round,line cap=round] ( 24.00, 78.55) -- ( 18.00, 78.55);

\path[draw=drawColor,line width= 0.4pt,line join=round,line cap=round] ( 24.00, 95.32) -- ( 18.00, 95.32);

\path[draw=drawColor,line width= 0.4pt,line join=round,line cap=round] ( 24.00,112.09) -- ( 18.00,112.09);

\path[draw=drawColor,line width= 0.4pt,line join=round,line cap=round] ( 24.00,128.86) -- ( 18.00,128.86);

\node[text=drawColor,rotate= 90.00,anchor=base,inner sep=0pt, outer sep=0pt, scale=  1.00] at ( 15.60, 28.24) {5};

\node[text=drawColor,rotate= 90.00,anchor=base,inner sep=0pt, outer sep=0pt, scale=  1.00] at ( 15.60, 45.01) {10};

\node[text=drawColor,rotate= 90.00,anchor=base,inner sep=0pt, outer sep=0pt, scale=  1.00] at ( 15.60, 78.55) {20};

\node[text=drawColor,rotate= 90.00,anchor=base,inner sep=0pt, outer sep=0pt, scale=  1.00] at ( 15.60,112.09) {30};

\path[draw=drawColor,line width= 0.4pt,line join=round,line cap=round] ( 24.00, 24.00) --
	(166.25, 24.00) --
	(166.25,138.54) --
	( 24.00,138.54) --
	( 24.00, 24.00);
\end{scope}
\begin{scope}
\path[clip] ( 24.00, 24.00) rectangle (166.25,138.54);
\definecolor{drawColor}{RGB}{0,0,0}

\path[draw=drawColor,line width= 0.4pt,line join=round,line cap=round] ( 29.27, 28.24) --
	( 43.90, 42.52) --
	( 58.54, 52.09) --
	( 73.17, 59.49) --
	( 87.81, 65.61) --
	(102.44, 70.87) --
	(117.08, 75.51) --
	(131.71, 79.68) --
	(146.34, 83.47) --
	(160.98, 86.95);

\path[draw=drawColor,line width= 0.4pt,dash pattern=on 4pt off 4pt ,line join=round,line cap=round] ( 29.27, 28.24) --
	( 43.90, 36.18) --
	( 58.54, 41.37) --
	( 73.17, 45.31) --
	( 87.81, 48.52) --
	(102.44, 51.26) --
	(117.08, 53.65) --
	(131.71, 55.78) --
	(146.34, 57.71) --
	(160.98, 59.48);
\end{scope}
\end{tikzpicture}
      \vspace{-0.5cm}
    \caption{\label{fig:m_n_two_level} $m_n$}
    \end{subfigure}
    \begin{subfigure}[b]{0.5 \textwidth}
\begin{tikzpicture}[x=1pt,y=1pt]
\definecolor{fillColor}{RGB}{255,255,255}
\path[use as bounding box,fill=fillColor,fill opacity=0.00] (0,0) rectangle (173.45,144.54);
\begin{scope}
\path[clip] (  0.00,  0.00) rectangle (173.45,144.54);
\definecolor{drawColor}{RGB}{0,0,0}

\path[draw=drawColor,line width= 0.4pt,line join=round,line cap=round] ( 43.90, 24.00) -- (160.98, 24.00);

\path[draw=drawColor,line width= 0.4pt,line join=round,line cap=round] ( 43.90, 24.00) -- ( 43.90, 18.00);

\path[draw=drawColor,line width= 0.4pt,line join=round,line cap=round] ( 73.17, 24.00) -- ( 73.17, 18.00);

\path[draw=drawColor,line width= 0.4pt,line join=round,line cap=round] (102.44, 24.00) -- (102.44, 18.00);

\path[draw=drawColor,line width= 0.4pt,line join=round,line cap=round] (131.71, 24.00) -- (131.71, 18.00);

\path[draw=drawColor,line width= 0.4pt,line join=round,line cap=round] (160.98, 24.00) -- (160.98, 18.00);

\node[text=drawColor,anchor=base,inner sep=0pt, outer sep=0pt, scale=  1.00] at ( 43.90,  8.40) {2000};

\node[text=drawColor,anchor=base,inner sep=0pt, outer sep=0pt, scale=  1.00] at ( 73.17,  8.40) {4000};

\node[text=drawColor,anchor=base,inner sep=0pt, outer sep=0pt, scale=  1.00] at (102.44,  8.40) {6000};

\node[text=drawColor,anchor=base,inner sep=0pt, outer sep=0pt, scale=  1.00] at (131.71,  8.40) {8000};

\path[draw=drawColor,line width= 0.4pt,line join=round,line cap=round] ( 24.00, 47.89) -- ( 24.00,125.86);

\path[draw=drawColor,line width= 0.4pt,line join=round,line cap=round] ( 24.00, 47.89) -- ( 18.00, 47.89);

\path[draw=drawColor,line width= 0.4pt,line join=round,line cap=round] ( 24.00, 73.88) -- ( 18.00, 73.88);

\path[draw=drawColor,line width= 0.4pt,line join=round,line cap=round] ( 24.00, 99.87) -- ( 18.00, 99.87);

\path[draw=drawColor,line width= 0.4pt,line join=round,line cap=round] ( 24.00,125.86) -- ( 18.00,125.86);

\node[text=drawColor,rotate= 90.00,anchor=base,inner sep=0pt, outer sep=0pt, scale=  1.00] at ( 15.60, 47.89) {0.6};

\node[text=drawColor,rotate= 90.00,anchor=base,inner sep=0pt, outer sep=0pt, scale=  1.00] at ( 15.60, 73.88) {0.8};

\node[text=drawColor,rotate= 90.00,anchor=base,inner sep=0pt, outer sep=0pt, scale=  1.00] at ( 15.60, 99.87) {1.0};

\node[text=drawColor,rotate= 90.00,anchor=base,inner sep=0pt, outer sep=0pt, scale=  1.00] at ( 15.60,125.86) {1.2};

\path[draw=drawColor,line width= 0.4pt,line join=round,line cap=round] ( 24.00, 24.00) --
	(166.25, 24.00) --
	(166.25,138.54) --
	( 24.00,138.54) --
	( 24.00, 24.00);
\end{scope}
\begin{scope}
\path[clip] ( 24.00, 24.00) rectangle (166.25,138.54);
\definecolor{drawColor}{RGB}{0,0,0}

\path[draw=drawColor,line width= 0.4pt,line join=round,line cap=round] ( 29.27, 84.59) --
	( 43.90, 76.03) --
	( 58.54, 74.42) --
	( 73.17, 73.88) --
	( 87.81, 73.71) --
	(102.44, 73.70) --
	(117.08, 73.77) --
	(131.71, 73.88) --
	(146.34, 74.00) --
	(160.98, 74.13);

\path[draw=drawColor,line width= 0.4pt,dash pattern=on 4pt off 4pt ,line join=round,line cap=round] ( 29.27, 84.54) --
	( 43.90, 95.21) --
	( 58.54,103.62) --
	( 73.17,110.19) --
	( 87.81,115.64) --
	(102.44,120.25) --
	(117.08,124.34) --
	(131.71,127.99) --
	(146.34,131.28) --
	(160.98,134.30);

\path[draw=drawColor,line width= 0.4pt,dash pattern=on 1pt off 3pt ,line join=round,line cap=round] ( 29.27, 84.41) --
	( 43.90, 53.87) --
	( 58.54, 44.62) --
	( 73.17, 39.59) --
	( 87.81, 36.29) --
	(102.44, 33.89) --
	(117.08, 32.04) --
	(131.71, 30.54) --
	(146.34, 29.30) --
	(160.98, 28.24);
\end{scope}
\end{tikzpicture}
       \vspace{-0.5cm}
    \caption{\label{fig:error_score_two_level}
      $\hat r_n^{-1/2} \delta_n(\theta_0)$
    }
    \end{subfigure}
    
    \vspace{0.5cm}
    
    \begin{subfigure}[b]{0.5 \textwidth}
\begin{tikzpicture}[x=1pt,y=1pt]
\definecolor{fillColor}{RGB}{255,255,255}
\path[use as bounding box,fill=fillColor,fill opacity=0.00] (0,0) rectangle (173.45,144.54);
\begin{scope}
\path[clip] (  0.00,  0.00) rectangle (173.45,144.54);
\definecolor{drawColor}{RGB}{0,0,0}

\path[draw=drawColor,line width= 0.4pt,line join=round,line cap=round] ( 43.90, 24.00) -- (160.98, 24.00);

\path[draw=drawColor,line width= 0.4pt,line join=round,line cap=round] ( 43.90, 24.00) -- ( 43.90, 18.00);

\path[draw=drawColor,line width= 0.4pt,line join=round,line cap=round] ( 73.17, 24.00) -- ( 73.17, 18.00);

\path[draw=drawColor,line width= 0.4pt,line join=round,line cap=round] (102.44, 24.00) -- (102.44, 18.00);

\path[draw=drawColor,line width= 0.4pt,line join=round,line cap=round] (131.71, 24.00) -- (131.71, 18.00);

\path[draw=drawColor,line width= 0.4pt,line join=round,line cap=round] (160.98, 24.00) -- (160.98, 18.00);

\node[text=drawColor,anchor=base,inner sep=0pt, outer sep=0pt, scale=  1.00] at ( 43.90,  8.40) {2000};

\node[text=drawColor,anchor=base,inner sep=0pt, outer sep=0pt, scale=  1.00] at ( 73.17,  8.40) {4000};

\node[text=drawColor,anchor=base,inner sep=0pt, outer sep=0pt, scale=  1.00] at (102.44,  8.40) {6000};

\node[text=drawColor,anchor=base,inner sep=0pt, outer sep=0pt, scale=  1.00] at (131.71,  8.40) {8000};

\path[draw=drawColor,line width= 0.4pt,line join=round,line cap=round] ( 24.00, 34.86) -- ( 24.00,129.70);

\path[draw=drawColor,line width= 0.4pt,line join=round,line cap=round] ( 24.00, 34.86) -- ( 18.00, 34.86);

\path[draw=drawColor,line width= 0.4pt,line join=round,line cap=round] ( 24.00, 50.66) -- ( 18.00, 50.66);

\path[draw=drawColor,line width= 0.4pt,line join=round,line cap=round] ( 24.00, 66.47) -- ( 18.00, 66.47);

\path[draw=drawColor,line width= 0.4pt,line join=round,line cap=round] ( 24.00, 82.28) -- ( 18.00, 82.28);

\path[draw=drawColor,line width= 0.4pt,line join=round,line cap=round] ( 24.00, 98.08) -- ( 18.00, 98.08);

\path[draw=drawColor,line width= 0.4pt,line join=round,line cap=round] ( 24.00,113.89) -- ( 18.00,113.89);

\path[draw=drawColor,line width= 0.4pt,line join=round,line cap=round] ( 24.00,129.70) -- ( 18.00,129.70);

\node[text=drawColor,rotate= 90.00,anchor=base,inner sep=0pt, outer sep=0pt, scale=  1.00] at ( 15.60, 34.86) {0.01};

\node[text=drawColor,rotate= 90.00,anchor=base,inner sep=0pt, outer sep=0pt, scale=  1.00] at ( 15.60, 66.47) {0.03};

\node[text=drawColor,rotate= 90.00,anchor=base,inner sep=0pt, outer sep=0pt, scale=  1.00] at ( 15.60, 98.08) {0.05};

\node[text=drawColor,rotate= 90.00,anchor=base,inner sep=0pt, outer sep=0pt, scale=  1.00] at ( 15.60,129.70) {0.07};

\path[draw=drawColor,line width= 0.4pt,line join=round,line cap=round] ( 24.00, 24.00) --
	(166.25, 24.00) --
	(166.25,138.54) --
	( 24.00,138.54) --
	( 24.00, 24.00);
\end{scope}
\begin{scope}
\path[clip] ( 24.00, 24.00) rectangle (166.25,138.54);
\definecolor{drawColor}{RGB}{0,0,0}

\path[draw=drawColor,line width= 0.4pt,line join=round,line cap=round] ( 29.27,134.30) --
	( 43.90, 68.66) --
	( 58.54, 52.95) --
	( 73.17, 45.65) --
	( 87.81, 41.16) --
	(102.44, 38.17) --
	(117.08, 36.24) --
	(131.71, 34.55) --
	(146.34, 33.27) --
	(160.98, 32.36);

\path[draw=drawColor,line width= 0.4pt,dash pattern=on 4pt off 4pt ,line join=round,line cap=round] ( 29.27,134.15) --
	( 43.90, 82.92) --
	( 58.54, 65.14) --
	( 73.17, 56.63) --
	( 87.81, 51.53) --
	(102.44, 47.48) --
	(117.08, 45.13) --
	(131.71, 43.15) --
	(146.34, 41.46) --
	(160.98, 40.12);

\path[draw=drawColor,line width= 0.4pt,dash pattern=on 1pt off 3pt ,line join=round,line cap=round] ( 29.27,132.15) --
	( 43.90, 57.30) --
	( 58.54, 43.88) --
	( 73.17, 37.89) --
	( 87.81, 34.89) --
	(102.44, 32.70) --
	(117.08, 31.16) --
	(131.71, 30.03) --
	(146.34, 29.02) --
	(160.98, 28.24);
\end{scope}
\end{tikzpicture}
       \vspace{-0.5cm}
    \caption{\label{fig:rmse_laplace_two_level}
      $\text{RMSE}(\tilde \theta_n)$}
     \end{subfigure}
     \begin{subfigure}[b]{0.5 \textwidth}
\begin{tikzpicture}[x=1pt,y=1pt]
\definecolor{fillColor}{RGB}{255,255,255}
\path[use as bounding box,fill=fillColor,fill opacity=0.00] (0,0) rectangle (173.45,144.54);
\begin{scope}
\path[clip] (  0.00,  0.00) rectangle (173.45,144.54);
\definecolor{drawColor}{RGB}{0,0,0}

\path[draw=drawColor,line width= 0.4pt,line join=round,line cap=round] ( 43.90, 24.00) -- (160.98, 24.00);

\path[draw=drawColor,line width= 0.4pt,line join=round,line cap=round] ( 43.90, 24.00) -- ( 43.90, 18.00);

\path[draw=drawColor,line width= 0.4pt,line join=round,line cap=round] ( 73.17, 24.00) -- ( 73.17, 18.00);

\path[draw=drawColor,line width= 0.4pt,line join=round,line cap=round] (102.44, 24.00) -- (102.44, 18.00);

\path[draw=drawColor,line width= 0.4pt,line join=round,line cap=round] (131.71, 24.00) -- (131.71, 18.00);

\path[draw=drawColor,line width= 0.4pt,line join=round,line cap=round] (160.98, 24.00) -- (160.98, 18.00);

\node[text=drawColor,anchor=base,inner sep=0pt, outer sep=0pt, scale=  1.00] at ( 43.90,  8.40) {2000};

\node[text=drawColor,anchor=base,inner sep=0pt, outer sep=0pt, scale=  1.00] at ( 73.17,  8.40) {4000};

\node[text=drawColor,anchor=base,inner sep=0pt, outer sep=0pt, scale=  1.00] at (102.44,  8.40) {6000};

\node[text=drawColor,anchor=base,inner sep=0pt, outer sep=0pt, scale=  1.00] at (131.71,  8.40) {8000};

\path[draw=drawColor,line width= 0.4pt,line join=round,line cap=round] ( 24.00, 29.28) -- ( 24.00,133.83);

\path[draw=drawColor,line width= 0.4pt,line join=round,line cap=round] ( 24.00, 29.28) -- ( 18.00, 29.28);

\path[draw=drawColor,line width= 0.4pt,line join=round,line cap=round] ( 24.00, 50.19) -- ( 18.00, 50.19);

\path[draw=drawColor,line width= 0.4pt,line join=round,line cap=round] ( 24.00, 71.10) -- ( 18.00, 71.10);

\path[draw=drawColor,line width= 0.4pt,line join=round,line cap=round] ( 24.00, 92.01) -- ( 18.00, 92.01);

\path[draw=drawColor,line width= 0.4pt,line join=round,line cap=round] ( 24.00,112.92) -- ( 18.00,112.92);

\path[draw=drawColor,line width= 0.4pt,line join=round,line cap=round] ( 24.00,133.83) -- ( 18.00,133.83);

\node[text=drawColor,rotate= 90.00,anchor=base,inner sep=0pt, outer sep=0pt, scale=  1.00] at ( 15.60, 29.28) {1.1};

\node[text=drawColor,rotate= 90.00,anchor=base,inner sep=0pt, outer sep=0pt, scale=  1.00] at ( 15.60, 71.10) {1.3};

\node[text=drawColor,rotate= 90.00,anchor=base,inner sep=0pt, outer sep=0pt, scale=  1.00] at ( 15.60,112.92) {1.5};

\path[draw=drawColor,line width= 0.4pt,line join=round,line cap=round] ( 24.00, 24.00) --
	(166.25, 24.00) --
	(166.25,138.54) --
	( 24.00,138.54) --
	( 24.00, 24.00);
\end{scope}
\begin{scope}
\path[clip] ( 24.00, 24.00) rectangle (166.25,138.54);
\definecolor{drawColor}{RGB}{0,0,0}

\path[draw=drawColor,line width= 0.4pt,line join=round,line cap=round] ( 29.27, 60.46) --
	( 43.90, 65.84) --
	( 58.54, 65.44) --
	( 73.17, 63.11) --
	( 87.81, 62.57) --
	(102.44, 63.18) --
	(117.08, 64.49) --
	(131.71, 67.19) --
	(146.34, 67.13) --
	(160.98, 68.26);

\path[draw=drawColor,line width= 0.4pt,dash pattern=on 4pt off 4pt ,line join=round,line cap=round] ( 29.27, 58.41) --
	( 43.90, 81.96) --
	( 58.54, 94.18) --
	( 73.17,102.35) --
	( 87.81,111.09) --
	(102.44,114.65) --
	(117.08,123.76) --
	(131.71,125.16) --
	(146.34,130.93) --
	(160.98,134.30);

\path[draw=drawColor,line width= 0.4pt,dash pattern=on 1pt off 3pt ,line join=round,line cap=round] ( 29.27, 58.09) --
	( 43.90, 47.37) --
	( 58.54, 38.64) --
	( 73.17, 36.65) --
	( 87.81, 33.72) --
	(102.44, 33.31) --
	(117.08, 31.94) --
	(131.71, 30.88) --
	(146.34, 29.36) --
	(160.98, 28.24);
\end{scope}
\end{tikzpicture}
       \vspace{-0.5cm}
    \caption{\label{fig:relative_rmse_two_level}
      $\text{RMSE}(\tilde \theta_n) / \text{RMSE}(\hat \theta_n)$}
     \end{subfigure}
     
    \vspace{0.5cm}
    
    \begin{subfigure}[b]{0.5 \textwidth}
\begin{tikzpicture}[x=1pt,y=1pt]
\definecolor{fillColor}{RGB}{255,255,255}
\path[use as bounding box,fill=fillColor,fill opacity=0.00] (0,0) rectangle (173.45,144.54);
\begin{scope}
\path[clip] (  0.00,  0.00) rectangle (173.45,144.54);
\definecolor{drawColor}{RGB}{0,0,0}

\path[draw=drawColor,line width= 0.4pt,line join=round,line cap=round] ( 43.90, 24.00) -- (160.98, 24.00);

\path[draw=drawColor,line width= 0.4pt,line join=round,line cap=round] ( 43.90, 24.00) -- ( 43.90, 18.00);

\path[draw=drawColor,line width= 0.4pt,line join=round,line cap=round] ( 73.17, 24.00) -- ( 73.17, 18.00);

\path[draw=drawColor,line width= 0.4pt,line join=round,line cap=round] (102.44, 24.00) -- (102.44, 18.00);

\path[draw=drawColor,line width= 0.4pt,line join=round,line cap=round] (131.71, 24.00) -- (131.71, 18.00);

\path[draw=drawColor,line width= 0.4pt,line join=round,line cap=round] (160.98, 24.00) -- (160.98, 18.00);

\node[text=drawColor,anchor=base,inner sep=0pt, outer sep=0pt, scale=  1.00] at ( 43.90,  8.40) {2000};

\node[text=drawColor,anchor=base,inner sep=0pt, outer sep=0pt, scale=  1.00] at ( 73.17,  8.40) {4000};

\node[text=drawColor,anchor=base,inner sep=0pt, outer sep=0pt, scale=  1.00] at (102.44,  8.40) {6000};

\node[text=drawColor,anchor=base,inner sep=0pt, outer sep=0pt, scale=  1.00] at (131.71,  8.40) {8000};

\path[draw=drawColor,line width= 0.4pt,line join=round,line cap=round] ( 24.00, 30.64) -- ( 24.00,132.19);

\path[draw=drawColor,line width= 0.4pt,line join=round,line cap=round] ( 24.00, 30.64) -- ( 18.00, 30.64);

\path[draw=drawColor,line width= 0.4pt,line join=round,line cap=round] ( 24.00, 50.95) -- ( 18.00, 50.95);

\path[draw=drawColor,line width= 0.4pt,line join=round,line cap=round] ( 24.00, 71.26) -- ( 18.00, 71.26);

\path[draw=drawColor,line width= 0.4pt,line join=round,line cap=round] ( 24.00, 91.57) -- ( 18.00, 91.57);

\path[draw=drawColor,line width= 0.4pt,line join=round,line cap=round] ( 24.00,111.88) -- ( 18.00,111.88);

\path[draw=drawColor,line width= 0.4pt,line join=round,line cap=round] ( 24.00,132.19) -- ( 18.00,132.19);

\node[text=drawColor,rotate= 90.00,anchor=base,inner sep=0pt, outer sep=0pt, scale=  1.00] at ( 15.60, 30.64) {0.65};

\node[text=drawColor,rotate= 90.00,anchor=base,inner sep=0pt, outer sep=0pt, scale=  1.00] at ( 15.60, 71.26) {0.75};

\node[text=drawColor,rotate= 90.00,anchor=base,inner sep=0pt, outer sep=0pt, scale=  1.00] at ( 15.60,111.88) {0.85};

\path[draw=drawColor,line width= 0.4pt,line join=round,line cap=round] ( 24.00, 24.00) --
	(166.25, 24.00) --
	(166.25,138.54) --
	( 24.00,138.54) --
	( 24.00, 24.00);
\end{scope}
\begin{scope}
\path[clip] ( 24.00, 24.00) rectangle (166.25,138.54);
\definecolor{drawColor}{RGB}{0,0,0}

\path[draw=drawColor,line width= 0.4pt,line join=round,line cap=round] ( 29.27, 88.76) --
	( 43.90, 89.45) --
	( 58.54, 87.55) --
	( 73.17, 89.33) --
	( 87.81, 90.39) --
	(102.44, 88.80) --
	(117.08, 87.18) --
	(131.71, 88.64) --
	(146.34, 89.17) --
	(160.98, 86.94);

\path[draw=drawColor,line width= 0.4pt,dash pattern=on 4pt off 4pt ,line join=round,line cap=round] ( 29.27, 87.50) --
	( 43.90, 70.40) --
	( 58.54, 64.47) --
	( 73.17, 58.22) --
	( 87.81, 49.28) --
	(102.44, 49.69) --
	(117.08, 41.08) --
	(131.71, 37.95) --
	(146.34, 33.04) --
	(160.98, 28.24);

\path[draw=drawColor,line width= 0.4pt,dash pattern=on 1pt off 3pt ,line join=round,line cap=round] ( 29.27, 91.57) --
	( 43.90,102.94) --
	( 58.54,108.87) --
	( 73.17,114.19) --
	( 87.81,112.40) --
	(102.44,116.71) --
	(117.08,114.15) --
	(131.71,116.55) --
	(146.34,119.23) --
	(160.98,120.53);

\path[draw=drawColor,line width= 0.4pt,line join=round,line cap=round] ( 29.27,132.79) --
	( 43.90,132.55) --
	( 58.54,130.97) --
	( 73.17,130.03) --
	( 87.81,131.13) --
	(102.44,131.94) --
	(117.08,130.60) --
	(131.71,132.55) --
	(146.34,133.49) --
	(160.98,132.14);

\path[draw=drawColor,line width= 0.4pt,dash pattern=on 4pt off 4pt ,line join=round,line cap=round] ( 29.27,132.06) --
	( 43.90,132.10) --
	( 58.54,132.79) --
	( 73.17,132.39) --
	( 87.81,133.36) --
	(102.44,134.30) --
	(117.08,133.97) --
	(131.71,132.10) --
	(146.34,132.31) --
	(160.98,133.57);

\path[draw=drawColor,line width= 0.4pt,dash pattern=on 1pt off 3pt ,line join=round,line cap=round] ( 29.27,133.16) --
	( 43.90,131.90) --
	( 58.54,131.82) --
	( 73.17,133.04) --
	( 87.81,129.67) --
	(102.44,132.75) --
	(117.08,133.61) --
	(131.71,132.43) --
	(146.34,133.85) --
	(160.98,132.96);
\end{scope}
\end{tikzpicture}
      \vspace{-0.5cm}
    \caption{\label{fig:coverage_laplace_two_level}
      Coverage of $90 \%$ confidence intervals}
    \end{subfigure}
          \begin{subfigure}[b]{0.5 \textwidth}
\begin{tikzpicture}[x=1pt,y=1pt]
\definecolor{fillColor}{RGB}{255,255,255}
\path[use as bounding box,fill=fillColor,fill opacity=0.00] (0,0) rectangle (173.45,144.54);
\begin{scope}
\path[clip] (  0.00,  0.00) rectangle (173.45,144.54);
\definecolor{drawColor}{RGB}{0,0,0}

\path[draw=drawColor,line width= 0.4pt,line join=round,line cap=round] ( 43.90, 24.00) -- (160.98, 24.00);

\path[draw=drawColor,line width= 0.4pt,line join=round,line cap=round] ( 43.90, 24.00) -- ( 43.90, 18.00);

\path[draw=drawColor,line width= 0.4pt,line join=round,line cap=round] ( 73.17, 24.00) -- ( 73.17, 18.00);

\path[draw=drawColor,line width= 0.4pt,line join=round,line cap=round] (102.44, 24.00) -- (102.44, 18.00);

\path[draw=drawColor,line width= 0.4pt,line join=round,line cap=round] (131.71, 24.00) -- (131.71, 18.00);

\path[draw=drawColor,line width= 0.4pt,line join=round,line cap=round] (160.98, 24.00) -- (160.98, 18.00);

\node[text=drawColor,anchor=base,inner sep=0pt, outer sep=0pt, scale=  1.00] at ( 43.90,  8.40) {2000};

\node[text=drawColor,anchor=base,inner sep=0pt, outer sep=0pt, scale=  1.00] at ( 73.17,  8.40) {4000};

\node[text=drawColor,anchor=base,inner sep=0pt, outer sep=0pt, scale=  1.00] at (102.44,  8.40) {6000};

\node[text=drawColor,anchor=base,inner sep=0pt, outer sep=0pt, scale=  1.00] at (131.71,  8.40) {8000};

\path[draw=drawColor,line width= 0.4pt,line join=round,line cap=round] ( 24.00, 36.41) -- ( 24.00,127.48);

\path[draw=drawColor,line width= 0.4pt,line join=round,line cap=round] ( 24.00, 36.41) -- ( 18.00, 36.41);

\path[draw=drawColor,line width= 0.4pt,line join=round,line cap=round] ( 24.00, 54.63) -- ( 18.00, 54.63);

\path[draw=drawColor,line width= 0.4pt,line join=round,line cap=round] ( 24.00, 72.84) -- ( 18.00, 72.84);

\path[draw=drawColor,line width= 0.4pt,line join=round,line cap=round] ( 24.00, 91.05) -- ( 18.00, 91.05);

\path[draw=drawColor,line width= 0.4pt,line join=round,line cap=round] ( 24.00,109.27) -- ( 18.00,109.27);

\path[draw=drawColor,line width= 0.4pt,line join=round,line cap=round] ( 24.00,127.48) -- ( 18.00,127.48);

\node[text=drawColor,rotate= 90.00,anchor=base,inner sep=0pt, outer sep=0pt, scale=  1.00] at ( 15.60, 36.41) {0.20};

\node[text=drawColor,rotate= 90.00,anchor=base,inner sep=0pt, outer sep=0pt, scale=  1.00] at ( 15.60, 72.84) {0.30};

\node[text=drawColor,rotate= 90.00,anchor=base,inner sep=0pt, outer sep=0pt, scale=  1.00] at ( 15.60,109.27) {0.40};

\path[draw=drawColor,line width= 0.4pt,line join=round,line cap=round] ( 24.00, 24.00) --
	(166.25, 24.00) --
	(166.25,138.54) --
	( 24.00,138.54) --
	( 24.00, 24.00);
\end{scope}
\begin{scope}
\path[clip] ( 24.00, 24.00) rectangle (166.25,138.54);
\definecolor{drawColor}{RGB}{0,0,0}

\path[draw=drawColor,line width= 0.4pt,line join=round,line cap=round] ( 29.27, 75.20) --
	( 43.90, 75.98) --
	( 58.54, 75.74) --
	( 73.17, 75.70) --
	( 87.81, 75.80) --
	(102.44, 75.95) --
	(117.08, 76.12) --
	(131.71, 76.32) --
	(146.34, 76.50) --
	(160.98, 76.69);

\path[draw=drawColor,line width= 0.4pt,dash pattern=on 4pt off 4pt ,line join=round,line cap=round] ( 29.27, 75.34) --
	( 43.90, 92.70) --
	( 58.54,102.99) --
	( 73.17,110.21) --
	( 87.81,115.89) --
	(102.44,120.63) --
	(117.08,125.97) --
	(131.71,128.24) --
	(146.34,131.42) --
	(160.98,134.30);

\path[draw=drawColor,line width= 0.4pt,dash pattern=on 1pt off 3pt ,line join=round,line cap=round] ( 29.27, 75.62) --
	( 43.90, 54.48) --
	( 58.54, 45.43) --
	( 73.17, 40.26) --
	( 87.81, 36.80) --
	(102.44, 34.27) --
	(117.08, 32.30) --
	(131.71, 30.71) --
	(146.34, 29.38) --
	(160.98, 28.24);
\end{scope}
\end{tikzpicture}
       \vspace{-0.5cm}
       \caption{\label{fig:tvd_two_level}
       $d_{TV}\{\tilde \pi(\theta | y), \pi(\theta | y)\}$}
      \end{subfigure}
          \caption{Comparison of exact and approximate likelihood
            inference, for
            a two-level model with $m_n$ observations on each of $n$ items,
       where $m_n = \min\{1, 5 + 4 (n^a - 1000^a)\}$, for 
       $a = 0.2$ (dashed lines), $0.25$ (solid lines) and $0.3$ (dotted lines).
       $\text{RMSE}(.)$ denotes root mean squared error.}
  \end{figure}
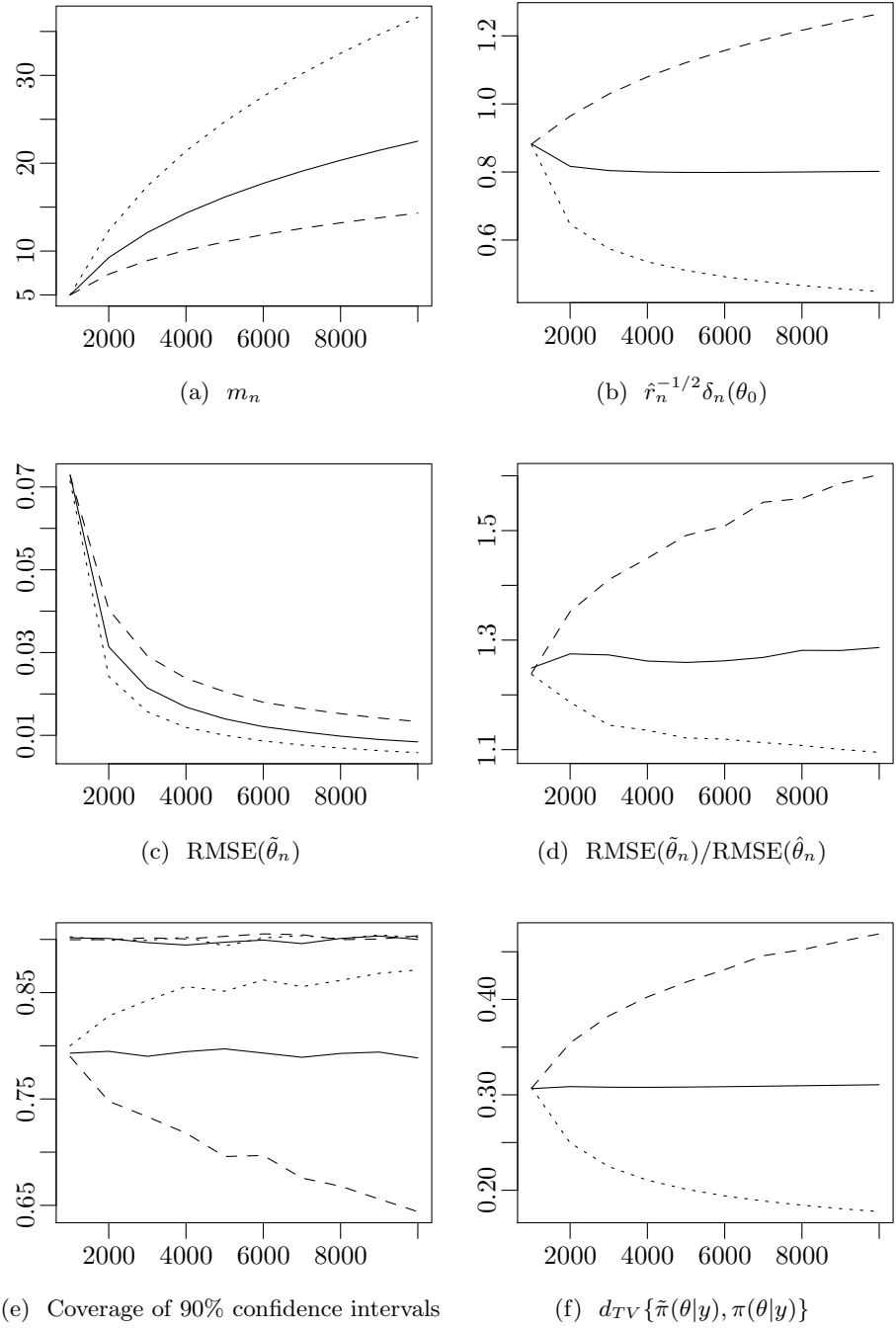
  
  To verify these results numerically, we
simulate $10000$ realizations from the model with $\theta_0 = 0.5$,
  various values
  of $n$ between $1000$ and $10000$,
  and $m_n = \min\{1, 5 + 4 (n^a - 1000^a)\}$, for
  $a = 0.2$, $0.25$ or $0.3 \null$. The three choices of $m_n$ are shown in Figure
  \ref{fig:m_n_two_level}.
  
  A very accurate approximation to the likelihood
  may be obtained by using
  adaptive Gaussian quadrature with $20$ quadrature points
  to approximate each of the univariate integrals $L_{(i)}(\theta) \null$.
  We use this approximation as a proxy for the exact likelihood $L(\theta) \null$.

  As $n \rightarrow \infty$, we have $r_n = O(n)$, but for
  smaller sample sizes $E(\norm{J_n(\hat \theta_n)})$ still grows with $m_n$.
  $E(\norm{J_n(\hat \theta_n)})$ may be approximated by
  $\hat r_n = \hat E(\norm{J_n(\hat \theta_n)})$,
  where $\hat E(.)$ the sample mean
  over the $10000$ realizations.
  The functional form of $m_n$ was chosen to make
  $\hat r_n^{-1/2} \delta_n(\theta_0)$
  approximately constant in the $a = 0.25$ case, as shown
  in Figure \ref{fig:error_score_two_level}. The same quantity
  grows with $n$ for $a = 0.2$ and shrinks with $n$ for $a = 0.3$.

  Figure \ref{fig:rmse_laplace_two_level} shows the
  root mean squared error in the Laplace estimator,
  and as expected, the estimator is consistent
  for all three choices of $a$.
  The root mean squared error of the Laplace
  estimator divided by that of the maximum likelihood
  estimator, shown
  in Figure \ref{fig:relative_rmse_two_level}, grows in the
  $a=0.2$ case, stays approximately constant if $a=0.25$,
  and shrinks towards $1$ if $a = 0.3$.

  Figure \ref{fig:coverage_laplace_two_level}
  shows the actual coverage of likelihood ratio type
  confidence intervals for $\theta$,
  of nominal $90 \%$ coverage. The upper three lines
  show the coverage of the intervals constructed
  using the exact likelihood, which have very close
  to nominal coverage for each $a$. The lower three lines
  show the coverage of the approximate likelihood intervals,
  which decreases with $n$ for $a = 0.2$, and increases
  towards the nominal $90 \%$ level for $a = 0.3$.

     Figure \ref{fig:tvd_two_level} shows the
     total variation distance between $\tilde \pi(\theta|y)$
     and $\pi(\theta|y)$, with prior $\pi(\theta) \propto 1/\theta \null$.
     The distance between the approximate
    posterior and the exact posterior grows in the
  $a=0.2$ case, stays approximately constant if $a=0.25$,
    and shrinks towards zero if $a = 0.3$.
  The behaviour for $a = 0.2$ refutes the conjecture of
  \cite{Rue2009} that the error in the approximate
  posterior distribution should shrink to zero
  provided that $m_n$ grows
  with $n$.

  \subsection{Reduced dependence approximation to the likelihood in an Ising model}

  \subsubsection{An Ising model}
  \label{sec:ising}
  We consider a simple Ising model for $n = rc$
  variables $y_i \in \{-1, 1\}$, arranged
  on an $r \times c$ lattice, with parameters $\theta = (\alpha, \beta)$,
  so that
  \[pr(Y= y; \theta) =
      \{Z_{r, c}(\theta)\}^{-1} \exp\{\alpha V_0(y) + \beta V_1(y)\},\]
      where $V_0(y) = \sum_{i} y_i$, and
      $V_1(y) = \sum_{i \sim j} y_i y_j$. Here
 $i \sim j$ indicates that $i$ and $j$ have an edge between
      them in the lattice, and $Z_{r, c}(\theta) = \sum_{y \in \{-1, 1\}^n}   \exp\{\alpha V_0(y) + \beta V_1(y)\}$
  is a normalizing constant. The likelihood function
  $L(\theta; y) = pr(Y = y; \theta)$ depends on
  $Z_{r, c}(\theta)$, and it is the
  computation of this normalizing constant which makes
  evaluation of the likelihood function difficult.
  By using variable elimination
  \citep[e.g.][]{Jordan2004},
  $Z_{r, c}(\theta)$
  may be computed at cost $O(rc 2^{\min\{r, c\}})$,
  which remains infeasibly expensive if both $r$ and $c$ are large.
  
  \subsubsection{Reduced dependence approximations}
  Many methods for approximating $Z_{r, c}(\theta)$ have been proposed.
  We will study properties of inference using the reduced
  dependence approximations introduced by \cite{Friel2009},
  a family of approximations
  controlled by
  a positive integer tuning parameter, which we call $k$.
  The approximation for fixed $k$ is
  $\tilde Z_{r, c}^{(k)}(\theta) = \{Z_{k, c}(\theta)\}^{r - k + 1} / \{Z_{k - 1, c}(\theta)\}^{r - k}.$
  
  We consider the case $r = c = m$, using a reduced dependence approximation
  at level $k$ to approximate  to the likelihood, giving
  $\tilde L^{(k)}_m(\beta) = \{\tilde Z_{m,  m}^{(k)}(\theta)\}^{-1} \exp\{\alpha V_0(y) + \beta V_1(y)\}.$
   The exact likelihood may be computed at cost
  $O(m^2 2^m)$, and the reduced
  dependence approximation at level $k$ at cost $O(m^2 + k m 2^{k})$. The aim
  is to understand how $k = k_m$ should vary with $m$ to give
  asymptotically valid inference as $m \rightarrow \infty.$
  The error in the log-likelihood,
  \[\epsilon^{(k)}_m(\beta) = \tilde \ell^{(k)}_m(\beta) - \ell_m(\beta) = \log Z_{m, m}(\beta) - \log \tilde Z_{m, m}^{(k)}(\beta),\]
  does not depend on the data $y$, so we do not need to consider
  the statements in probability: all of the errors are
  deterministic in this case.

  \subsubsection{A special case where the exact likelihood is available}
 
  If $\alpha = 0$, and the lattice has periodic boundary conditions,
  so that the top row of variables are joined to the bottom row,
  and the left row joined to the right, \cite{Kaufman1949}
  provides a relatively simple expression for $Z_{r,c}(0, \beta)$,
  so that it
  is possible to compute the likelihood exactly, even for large
  lattices.

  We restrict the parameter space to $\alpha = 0$, and assume
   $\beta \in [0, 0.43]$.
  This guarantees that
  $\beta < \beta_c = \log(1+\sqrt{2})/2 \approx 0.44$, where
  $\beta_c$ is a critical value at which the behaviour of the
  Ising model suddenly changes, so that for $\beta > \beta_c$ large
  areas of all plus ones or all minus ones are observed.
  If $\beta_0 = \beta_c$, the maximum likelihood estimator
  may not have a normal limiting distribution, so our results
  do not apply to this case.

  The information provided by the data about
  $\beta$ grows at rate $r_m = m^2 \null$.
  In the supplementary material, we show that
  if $k \rightarrow \infty$ as $m \rightarrow \infty$, then
  $\delta^{(k)}_{m}(\beta) = O\{m^2 k \exp(-b_\beta k)\} + o(1)$, where
   $b_\beta = 2 \cosh^{-1}\{-1 + \cosh(2 \beta)^2 / \sinh(2 \beta)\}.$

  For any choice of $k$ which grows with $m$,
  $\sup_{\beta \in [0, 0.43]} \delta_m^{(k)}(\beta) = o(m^2)$, so the approximate
  likelihood estimator will be consistent.

  In order to meet the conditions of Theorem \ref{thm:normality},
  $k = k_m$ should be chosen so that
  $\delta_m^{(k)}\{B_t(\beta_0)\} = o(m)$.
  For any
  $t < \beta_c - \beta_0$, for sufficiently large $m$,
  $\delta_m^{(k)}\{B_t(\beta_0)\} = \delta_m^{(k)}(\beta_0 + t) = O\{m^2 k \exp(-b_{\beta_0+t} k)\} + o(1)$. Since $b_{\beta}$ is a continuous
  function of $\beta$,
  any $k$ such that
  $m^2 k \exp(-b_{\beta_0} k) = o(m)$ will meet this condition
  for sufficiently small $t$.
  This may be achieved by taking $k_m = c_{\beta_0} \log{m}$,
  for any $c_{\beta_0} > b_{\beta_0}^{-1}$. Figure \ref{fig:b_beta_inv}
  shows how $b_{\beta}^{-1}$ varies with $\beta$.

  For any $k \rightarrow \infty$ as $m \rightarrow \infty$,
  $\gamma_m^{(k)}\{B_t(\beta_0)\} = o(m^2)$, so
  the reduced dependence approximation
  to the
  likelihood with $k_m = c_{\beta_0} \log{m}$ will provide
  asymptotically valid inference for any $c_{\beta_0} > b_{\beta_0}^{-1}$.
  The cost of computing this approximation is
  $O(m \log{m} \, m^{c_{\beta_0} \log{2}}) < O(m^{c_{\beta_0} \log{2}+2}),$
  polynomial in the size of the model,
  but increasing rapidly as $\beta_0$ approaches the critical value. 

 \begin{figure}
    \begin{subfigure}[b]{0.5 \textwidth}
\begin{tikzpicture}[x=1pt,y=1pt]
\definecolor{fillColor}{RGB}{255,255,255}
\path[use as bounding box,fill=fillColor,fill opacity=0.00] (0,0) rectangle (187.90,158.99);
\begin{scope}
\path[clip] (  0.00,  0.00) rectangle (187.90,158.99);
\definecolor{drawColor}{RGB}{0,0,0}

\path[draw=drawColor,line width= 0.4pt,line join=round,line cap=round] ( 36.74, 31.20) -- (178.95, 31.20);

\path[draw=drawColor,line width= 0.4pt,line join=round,line cap=round] ( 36.74, 31.20) -- ( 36.74, 27.60);

\path[draw=drawColor,line width= 0.4pt,line join=round,line cap=round] ( 57.05, 31.20) -- ( 57.05, 27.60);

\path[draw=drawColor,line width= 0.4pt,line join=round,line cap=round] ( 77.37, 31.20) -- ( 77.37, 27.60);

\path[draw=drawColor,line width= 0.4pt,line join=round,line cap=round] ( 97.69, 31.20) -- ( 97.69, 27.60);

\path[draw=drawColor,line width= 0.4pt,line join=round,line cap=round] (118.00, 31.20) -- (118.00, 27.60);

\path[draw=drawColor,line width= 0.4pt,line join=round,line cap=round] (138.32, 31.20) -- (138.32, 27.60);

\path[draw=drawColor,line width= 0.4pt,line join=round,line cap=round] (158.63, 31.20) -- (158.63, 27.60);

\path[draw=drawColor,line width= 0.4pt,line join=round,line cap=round] (178.95, 31.20) -- (178.95, 27.60);

\node[text=drawColor,anchor=base,inner sep=0pt, outer sep=0pt, scale=  1.00] at ( 36.74, 15.60) {0.10};

\node[text=drawColor,anchor=base,inner sep=0pt, outer sep=0pt, scale=  1.00] at ( 77.37, 15.60) {0.20};

\node[text=drawColor,anchor=base,inner sep=0pt, outer sep=0pt, scale=  1.00] at (118.00, 15.60) {0.30};

\node[text=drawColor,anchor=base,inner sep=0pt, outer sep=0pt, scale=  1.00] at (158.63, 15.60) {0.40};

\path[draw=drawColor,line width= 0.4pt,line join=round,line cap=round] ( 31.20, 35.71) -- ( 31.20,138.74);

\path[draw=drawColor,line width= 0.4pt,line join=round,line cap=round] ( 31.20, 35.71) -- ( 27.60, 35.71);

\path[draw=drawColor,line width= 0.4pt,line join=round,line cap=round] ( 31.20, 61.47) -- ( 27.60, 61.47);

\path[draw=drawColor,line width= 0.4pt,line join=round,line cap=round] ( 31.20, 87.22) -- ( 27.60, 87.22);

\path[draw=drawColor,line width= 0.4pt,line join=round,line cap=round] ( 31.20,112.98) -- ( 27.60,112.98);

\path[draw=drawColor,line width= 0.4pt,line join=round,line cap=round] ( 31.20,138.74) -- ( 27.60,138.74);

\node[text=drawColor,anchor=base east,inner sep=0pt, outer sep=0pt, scale=  1.00] at ( 25.20, 32.27) {0};

\node[text=drawColor,anchor=base east,inner sep=0pt, outer sep=0pt, scale=  1.00] at ( 25.20, 58.02) {5};

\node[text=drawColor,anchor=base east,inner sep=0pt, outer sep=0pt, scale=  1.00] at ( 25.20, 83.78) {10};

\node[text=drawColor,anchor=base east,inner sep=0pt, outer sep=0pt, scale=  1.00] at ( 25.20,109.54) {15};

\node[text=drawColor,anchor=base east,inner sep=0pt, outer sep=0pt, scale=  1.00] at ( 25.20,135.29) {20};

\path[draw=drawColor,line width= 0.4pt,line join=round,line cap=round] ( 31.20, 31.20) --
	(180.70, 31.20) --
	(180.70,152.99) --
	( 31.20,152.99) --
	( 31.20, 31.20);
\end{scope}
\begin{scope}
\path[clip] (  0.00,  0.00) rectangle (187.90,158.99);
\definecolor{drawColor}{RGB}{0,0,0}

\node[text=drawColor,anchor=base,inner sep=0pt, outer sep=0pt, scale=  1.00] at (105.95,  3.60) {$\beta$};

\node[text=drawColor,rotate= 90.00,anchor=base,inner sep=0pt, outer sep=0pt, scale=  1.00] at ( 10.80, 92.10) {$b_{\beta}^{-1}$};
\end{scope}
\begin{scope}
\path[clip] ( 31.20, 31.20) rectangle (180.70,152.99);
\definecolor{drawColor}{RGB}{0,0,0}

\path[draw=drawColor,line width= 0.4pt,line join=round,line cap=round] ( 36.74, 36.93) --
	( 38.77, 36.97) --
	( 40.80, 37.00) --
	( 42.83, 37.04) --
	( 44.86, 37.08) --
	( 46.90, 37.11) --
	( 48.93, 37.15) --
	( 50.96, 37.19) --
	( 52.99, 37.23) --
	( 55.02, 37.27) --
	( 57.05, 37.32) --
	( 59.08, 37.36) --
	( 61.12, 37.40) --
	( 63.15, 37.45) --
	( 65.18, 37.50) --
	( 67.21, 37.55) --
	( 69.24, 37.60) --
	( 71.27, 37.65) --
	( 73.31, 37.70) --
	( 75.34, 37.76) --
	( 77.37, 37.82) --
	( 79.40, 37.88) --
	( 81.43, 37.94) --
	( 83.46, 38.01) --
	( 85.50, 38.07) --
	( 87.53, 38.14) --
	( 89.56, 38.22) --
	( 91.59, 38.30) --
	( 93.62, 38.38) --
	( 95.65, 38.46) --
	( 97.69, 38.55) --
	( 99.72, 38.64) --
	(101.75, 38.74) --
	(103.78, 38.85) --
	(105.81, 38.96) --
	(107.84, 39.07) --
	(109.87, 39.20) --
	(111.91, 39.33) --
	(113.94, 39.47) --
	(115.97, 39.62) --
	(118.00, 39.78) --
	(120.03, 39.95) --
	(122.06, 40.13) --
	(124.10, 40.33) --
	(126.13, 40.55) --
	(128.16, 40.78) --
	(130.19, 41.03) --
	(132.22, 41.31) --
	(134.25, 41.61) --
	(136.29, 41.95) --
	(138.32, 42.32) --
	(140.35, 42.74) --
	(142.38, 43.21) --
	(144.41, 43.74) --
	(146.44, 44.34) --
	(148.48, 45.04) --
	(150.51, 45.84) --
	(152.54, 46.80) --
	(154.57, 47.94) --
	(156.60, 49.33) --
	(158.63, 51.07) --
	(160.66, 53.29) --
	(162.70, 56.23) --
	(164.73, 60.31) --
	(166.76, 66.38) --
	(168.79, 76.30) --
	(170.82, 95.51) --
	(172.85,148.48);

\path[draw=drawColor,line width= 0.4pt,dash pattern=on 4pt off 4pt ,line join=round,line cap=round] (175.16, 31.20) -- (175.16,152.99);
\end{scope}
\end{tikzpicture}
      \vspace{-0.5cm}
      \caption{\label{fig:b_beta_inv} $b_\beta^{-1}$, for a restricted
        Ising model with
        $\alpha = 0$.}
    \end{subfigure}
    \begin{subfigure}[b]{0.5 \textwidth}
      \input{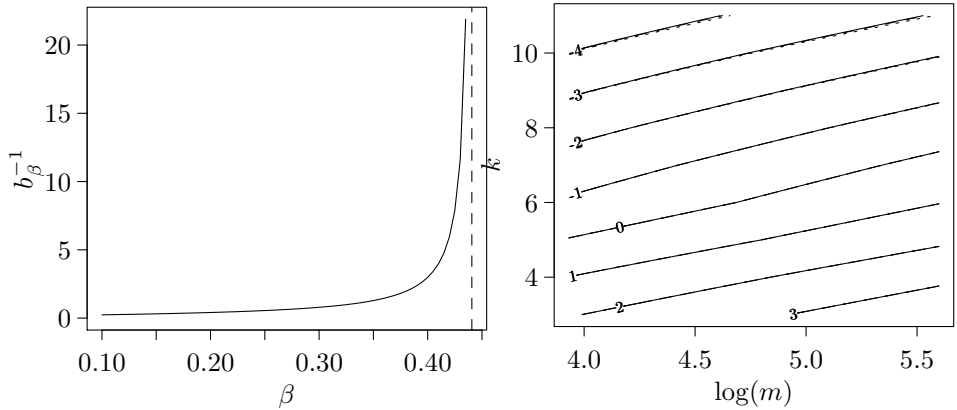}
       \vspace{-0.5cm}
    \caption{\label{fig:delta_rescaled_ising_0pt1_0pt3} Contour plot of $\log \{m^{-1} \delta_m^{(k)}(0.1, 0.3) \}$ against $\log{m}$ and $k$.}
    \end{subfigure}
    \caption{Inference for an Ising model with a reduced
    dependence approximation to the likelihood.}
    \end{figure}

 \subsubsection{Extension to the more general case}
\label{sec:ising_field}
 
  If $\alpha \not = 0$, no simple expression for
  $Z_{r,c}(\alpha, \beta)$ is known.
  Instead, we investigate the behaviour of
  $m^{-1} \delta^{(k)}_{m}(\alpha, \beta)$ numerically, by
  using $\tilde \ell^{(K)}_{m}$ for some large $K$ as a proxy
  for the exact log-likelihood. Care is needed to choose
  $K$ sufficiently large to ensure that the results are not
  sensitive to the choice of $K$.

  Taking $K = 16$, a contour plot of this approximation to
  $\log \{m^{-1} \delta^{(k)}_{m}(0.1, 0.3) \}$
  against $\log{m}$ and $k$
  is shown in Figure \ref{fig:delta_rescaled_ising_0pt1_0pt3},
  for $k = 2, \ldots, 12$, $m = 50, \ldots 300 \null$.
   The same plots with $K=15$ (dashed) and $K=14$ (dot-dash) are
  overlaid, and the differences are barely visible at this scale.
  Given this stability, it seems reasonable to assume a
  contour plot of the exact $\log \{m^{-1} \delta^{(k)}_{m}(0.1, 0.3) \}$
  would look very similar to this.
  To obtain asymptotically valid inference,
  $k_m$ should be chosen
  so that this rescaled error shrinks with $m$, which seems
  to occur if
  $k_m$ grows at rate $c(0.1, 0.3) \log{m}$, for
  $c(0.1, 0.3)$ larger than about $1.5$. This pattern of
  behaviour is very similar to the $\alpha = 0$ case.
  In both cases, reduced dependence approximations with an
  appropriate choice of $k$ will
  give asymptotically valid inference at cost polynomial
  in the size of the model, in contrast with
  the exponential cost of computing the likelihood exactly.

  \section{Discussion}

  The results obtained here can also be applied
  to other approximations to the likelihood
  and to other types of model.
  The conditions on the approximate likelihood,
  such as showing that $\delta_n^{\infty}\{B_t(\theta_0)\} = o_p(r_n)$,
  may be difficult to verify in practice,
  as the exact likelihood is assumed unavailable.
  If the approximation to the likelihood is a truncation of
  a series expansion for the exact likelihood,
  as is the case for the Laplace approximation, the
  conditions can be checked by examining the contribution
  from higher-order terms in the expansion.
  In other cases, 
  it may only be possible to investigate
  the size of the errors numerically, by using
  used a more accurate and expensive approximation
  to the likelihood as a proxy for the exact likelihood,
  as described in Section \ref{sec:ising_field}
  for the Ising model example.

Many approximation methods have a tuning parameter,
$k$ say, where increasing $k$ allows computation of
a more accurate likelihood approximation at increased cost.
The reduced dependence approximation at level $k$ described
in Section \ref{sec:ising} is one example
of this.
In order for the approximate likelihood inference to be
close to the exact likelihood inference,
$k = k_n$ should be allowed
to vary with $n$.
Given an understanding of how the error
in the score function varies with $k$ and $n$,
the results obtained here could be applied to determine
how $k_n$ should scale with $n$.
This has the potential to allow the construction of approximate
likelihoods 
which match the inference with the exact likelihood
closely for all $n$,
but which scale well to large data sizes.

\section*{Acknowledgements}
I am grateful to David Firth, Nancy Reid, Cristiano Varin and two
referees for helpful comments which have greatly improved this paper.
This work was supported by the
Engineering and Physical Sciences Research Council
through Intractable Likelihood: New Challenges from Modern Applications
[grant number EP/K014463/1].

  \appendix
  
\section*{Appendix: proofs of results}
\label{sec:proofs}

\subsection*{Proof of Theorem \ref{thm:consistency}}
  \begin{proof}
  We apply Theorem 5.9 of \cite{VanDerVaart2000}, with
  $\Psi_n(\theta) = r_n^{-1} \tilde u_n(\theta)$
  and $\Psi(\theta) = \bar u(\theta)$.
  Then
  \begin{align*}
    \sup_{\theta \in \Theta} \norm{\Psi_n(\theta) - \Psi(\theta)} &=
    \sup_{\theta \in \Theta} \norm{\bar u_n(\theta) + r_n^{-1} \nabla_\theta \epsilon_n(\theta) - \bar u(\theta)} \\
    &\leq \sup_{\theta \in \Theta} \norm{\bar u_n(\theta) - \bar u(\theta)} +r_n^{-1} \delta_n \\
    &= o_p(1)
  \end{align*}
  The conditions of Theorem 5.9 of \cite{VanDerVaart2000}
  are met, so since $\Psi_n(\tilde \theta_n) = 0$, we have that
  $\tilde \theta_n \rightarrow \theta_0$ in probability,
  as $n \rightarrow \infty$.
  \end{proof}

  \subsection*{Proof of Lemma \ref{lemma:distance_estimators}}
\begin{proof}

 Taking a Taylor expansion of $\bar u_n$ about $\hat \theta_n$, 
 \[\bar u_n(\theta) = \bar u_n (\hat \theta_n) - (\theta - \hat \theta_n)^T \bar J_n(\theta_n^*)  = - (\theta - \hat \theta_n)^T \bar J_n(\theta_n^*),\]
 for some $\theta^*_n$ between $\theta$
 and $\hat \theta_n$.
Define
 $\bar{\tilde{u}}_n(\theta)= r_n^{-1} \tilde u_n(\theta)$. Then
 \[\bar{\tilde{u}}_n(\theta) = \bar u_n(\theta) +
 r_n^{-1} \nabla_\theta \epsilon_n(\theta)
= - (\theta - \hat \theta_n)^T \bar J_n(\theta_n^*) + r_n^{-1} \nabla_\theta \epsilon_n(\theta), \]
so any $\tilde \theta_n$ solving $\bar{\tilde{u}}_n(\tilde \theta_n) = 0$
solves
\[\tilde \theta_n - \hat \theta_n = \bar J_n^{-1}(\theta^*_n) r_n^{-1} \nabla_\theta\epsilon_n(\tilde \theta_n),\]
 for some $\theta^*_n$ between $\tilde \theta_n$
 and $\hat \theta_n$.

 But $\theta^*_n$
 is a consistent estimator of $\theta$, because
 $\tilde \theta_n$ is by Theorem \ref{thm:consistency},
 so
 $\bar J_n(\theta^*_n) \rightarrow I(\theta_0)$
 in probability.
So 
\[\tilde \theta_n - \hat \theta_n = O_p(r_n^{-1} \delta_n(\tilde \theta_n)) .\]

Write $A_n$ for the event
$\{\tilde \theta_n \in B_{t}(\theta_0) \}$.
Then $\text{pr}(A_n) \rightarrow 1$ as $n \rightarrow \infty$
since $\tilde \theta_n$ is consistent, and
conditional on $A_n$, 
\[\tilde \theta_n - \hat \theta_n = O_p\left[r_n^{-1} \delta_n^\infty\{B_t(\theta_0)\}\right] = o_p(r_n^{-1} a_n).\]
For any $\epsilon > 0$,
\begin{align*}
\text{pr}(&\norm{\tilde \theta_n - \hat \theta_n} \geq \epsilon a_n r_n^{-1}) \\
&= \text{pr}(\norm{\tilde \theta_n - \hat \theta_n} \geq \epsilon a_n r_n^{-1} | A_n) \text{pr}(A_n) + \text{pr}(\norm{\tilde \theta_n - \hat \theta_n} \geq \epsilon a_n r_n^{-1} | A_n^C) \text{pr}(A_n^C) \\
& \leq \text{pr}(\norm{\tilde \theta_n - \hat \theta_n} \geq \epsilon a_n r_n^{-1} | A_n) + \text{pr}(A_n^C) \\
&\rightarrow 0
\end{align*}
as $n \rightarrow \infty$.
\end{proof}

\subsection*{Proof of Theorem \ref{thm:normality}}
\begin{proof}
Applying Lemma \ref{lemma:distance_estimators} 
with $a_n = r_n^{1/2}$,
\[\tilde \theta_n - \hat \theta_n = o_p(r_n^{1/2} r_n^{-1}) = o_p(r_n^{-1/2}).\]
So
\[r_n^{1/2} (\tilde \theta_n - \theta_0) = r_n^{1/2} (\hat \theta_n - \theta_0) + o_p(1) \rightarrow N(0, I(\theta_0)^{-1})\]
in distribution, as required.
\end{proof}

\subsection*{Proof of Theorem \ref{thm:hypothesis}}
\begin{proof}

We have
\begin{align*}
  (\tilde \Lambda_n -  \Lambda_n)/2 &=
 \{\tilde \ell_n(\tilde \theta_n) - \tilde \ell_n(\tilde \theta_n^R)\} -
  \{\ell_n(\hat \theta_n) - \ell_n(\hat \theta_n^R)\}  \\
 &=  \{\tilde \ell_n(\tilde \theta_n) - \tilde \ell_n(\hat \theta_n)\}
  + \{\tilde \ell_n(\hat \theta_n) - \ell_n(\hat \theta_n)\}
 + \{\tilde \ell_n(\hat \theta_n^R) - \tilde \ell_n(\tilde \theta_n^R)\}
 + \{\ell_n(\hat \theta_n^R) - \tilde \ell_n(\hat \theta_n^R)\}.
\end{align*}
Taking a Taylor expansion of the first term,
  \begin{align*}
  \tilde \ell_n(\hat \theta_n) - \tilde \ell_n(\tilde \theta_n) &=
  (\hat \theta_n- \tilde \theta_n)^T \tilde u_n(\tilde \theta_n)
  + (\hat \theta_n - \tilde \theta_n)^T \tilde J_n(\bar \theta_n) (\hat \theta_n - \tilde \theta_n) \\
  &= (\hat \theta_n - \tilde \theta_n)^T \tilde J_n(\bar \theta_n) (\hat \theta_n - \tilde \theta_n)
  \end{align*}
for some $\bar \theta_n$ between $\hat \theta_n$ and
$\tilde \theta_n$, so
$\tilde \ell_n(\hat \theta_n) - \tilde \ell_n(\tilde \theta_n)= o_p(1)$,
  since
    $\norm{\hat \theta_n - \tilde \theta_n} = o_p(r_n^{-1/2})$
  and $\tilde J_n(\bar \theta_n) = J_n(\bar \theta_n) + o_p(r_n) = O_p(r_n)$.
  Similarly,
  $\tilde \ell_n(\hat \theta_n^R) - \tilde \ell_n(\tilde \theta_n^R)= o_p(1)$,
so
\begin{align*}
  (\tilde \Lambda_n -  \Lambda_n)/2
   &=  \epsilon_n(\hat \theta_n) - \epsilon_n(\hat \theta_n^R) + o_p(1)\\
  &= (\hat \theta_n - \hat \theta_n^R)^T \nabla_\theta \epsilon_n(\theta_n^*) + o_p(1)
\end{align*}
for some $\theta_n^*$ between $\hat \theta_n$ and $\hat \theta_n^R$.
But $\hat \theta_n - \hat \theta_n^R = O_p(r_n^{-1/2})$, and,
for sufficiently large $n$,
$\norm{\nabla_\theta \epsilon_n(\theta_n^*)} \leq \delta_n^\infty\{B_t(\theta_0)\},$
so
$\tilde \Lambda_n -  \Lambda_n = o_p(1)$,
as required.

\end{proof}

\subsection*{Results needed to prove Theorem \ref{thm:posterior}}

  In order to prove Theorem \ref{thm:posterior}, it is helpful
  to first consider properties of
  inference with the penalized loglikelihood
  $\ell_n^\pi(\theta) = \ell(\theta) + \log \pi(\theta)$
  with log-prior penalty. We write
  $\hat \theta_n^\pi$ for the corresponding penalized likelihood
  estimator, which is the posterior mode.
  Similarly, write $A^\pi$ for the version of the quantity $A$
  computed with $\ell_n^\pi(.)$ in place of $\ell_n(.)$,
  and $\tilde A^\pi$ for the approximate version.
 Since
 $\epsilon_n^\pi(\theta) = \epsilon_n(\theta)$, all
 the error terms remain unchanged.

Under the regularity conditions assumed on the model,
 the penalized likelihood
 estimator $\hat \theta_n^\pi$ will be consistent,
 and
 for any $b_n = o_p(r_n)$, the posterior probability
 that $\theta \in B_{b_n^{-1/2}}(\theta_0)$
 will tend to one as $n \rightarrow \infty$.

To prove Theorem \ref{thm:posterior},
we will use the following lemma, which says that
the error in the penalized log-likelihood ratio
may be approximated
in terms of the error in the score function.
\begin{lemma}
  \label{lemma:LR_diff}
 Suppose that $\delta_n^\infty = o_p(r_n)$,
 and that 
there exists $t > 0$ such that
$\delta_n^\infty\{B_t(\theta_0)\} = o_p(r_n^{1/2})$.
  Then $\hat \theta_n^\pi$ and $\tilde \theta_n^\pi$ are
  consistent estimators of $\theta_0$,
  and $\hat \theta_n^\pi - \tilde \theta_n^\pi = o_p(r_n^{-1/2})$.
  
  If $\gamma_n^\infty\{B_t(\theta_0)\} = o_p(b_n)$, for
  some $b_n = o(r_n)$, then
  \[\sup_{\theta \in B_{b_n^{-1/2}}(\hat \theta_n^\pi)}
  \Big| \big[
        \{\tilde \ell_n^\pi(\tilde \theta_n^\pi) - \tilde \ell_n^\pi(\theta)\}
        -\{\ell_n^\pi(\hat \theta_n^\pi) - \ell_n^\pi(\theta)\}\big]
        - \big[(\hat \theta_n^\pi - \theta)^T \nabla_\theta \epsilon_n(\theta)\big] \Big| = o_p(1).\]
  
\end{lemma}

\begin{proof}
  That $\hat \theta_n^\pi - \tilde \theta_n^\pi = o_p(r_n^{-1/2})$
  follows by a similar argument to the proof of Theorem \ref{thm:normality}.
Write
  \begin{align*}
C_n(\theta) &= \{\tilde \ell_n^\pi(\tilde \theta_n^\pi) - \tilde \ell_n^\pi(\theta)\}
 - \{\ell_n^\pi(\hat \theta_n^\pi) - \ell_n^\pi(\theta)\}\\
 &= \{\tilde \ell_n^\pi(\hat \theta_n^\pi) - \ell_n^\pi(\hat \theta_n^\pi)\}
 -  \{\tilde \ell_n^\pi(\theta) - \ell_n^\pi(\theta)\}
 - \{\tilde \ell_n^\pi(\hat \theta_n^\pi) - \tilde \ell_n^\pi(\tilde \theta_n^\pi)\}
 \\
 &= \epsilon_n(\hat \theta_n^\pi) - \epsilon_n(\theta)
 - \{\tilde \ell_n^\pi(\hat \theta_n^\pi) - \tilde \ell_n^\pi(\tilde \theta_n^\pi)\}.
 \end{align*}
 Then
 \[\tilde \ell_n(\hat \theta_n^\pi) - \tilde \ell_n(\tilde \theta_n^\pi) 
 = (\hat \theta_n^\pi - \tilde \theta_n^\pi)^T \tilde J_n^\pi(\bar \theta_n) (\hat \theta_n^\pi - \tilde \theta_n^\pi)\]
 for some $\bar \theta_n$ between $\hat \theta_n^\pi$ and
 $\tilde \theta_n^\pi$,
 and
 \[ \epsilon_n(\hat \theta_n^\pi) - \epsilon_n(\theta)
 =  (\hat \theta_n^\pi - \theta)^T \nabla_\theta \epsilon_n(\theta)
  + (\hat \theta_n^\pi - \theta)^T \nabla_\theta^T \nabla_\theta \epsilon_n(\theta_n^*(\theta)) (\hat \theta_n^\pi - \theta) \]
 for some $\theta_n^*(\theta)$ between $\hat \theta_n^\pi$ and $\theta$.
Write
\begin{align*}
  D_n(\theta) &= C_n(\theta) -(\hat \theta_n^\pi - \theta)^T \nabla_\theta \epsilon_n(\theta) \\
  &= (\hat \theta_n^\pi - \theta)^T \nabla_\theta^T \nabla_\theta \epsilon_n(\theta_n^*(\theta)) (\hat \theta_n^\pi - \theta) -  (\hat \theta_n^\pi - \tilde \theta_n^\pi)^T \tilde J_n^\pi(\bar \theta_n) (\hat \theta_n^\pi - \tilde \theta_n^\pi).
 \end{align*}
Then
$\sup_{\theta \in B_{b_n^{-1/2}}(\hat \theta_n^\pi)}  \big|D_n(\theta) \big|$
may be expressed as
 \begin{align*}
   & \sup_{\theta \in B_{b_n^{-1/2}}(\hat \theta_n^\pi)} \Big\{
   \big|(\hat \theta_n^\pi - \theta)^T \nabla_\theta^T \nabla_\theta \epsilon_n(\theta_n^*(\theta)) (\hat \theta_n^\pi - \theta) -  (\hat \theta_n^\pi - \tilde \theta_n^\pi)^T \tilde J_n^\pi(\bar \theta_n) (\hat \theta_n^\pi - \tilde \theta_n^\pi) \big| \Big\} \\
   &\leq \sup_{\theta \in B_{b_n^{-1/2}}(\hat \theta_n^\pi)} \Big\{ \big|(\hat \theta_n^\pi - \theta)^T \nabla_\theta^T \nabla_\theta \epsilon_n(\theta_n^*(\theta)) (\hat \theta_n^\pi - \theta) \big | \Big\} +
   \big| (\hat \theta_n^\pi - \tilde \theta_n^\pi)^T \tilde J_n^\pi(\bar \theta_n) (\hat \theta_n^\pi - \tilde \theta_n^\pi) \big| \\
   &\leq b_n^{-1} \delta_n^\infty\{B_t(\theta_0)\} + o_p(1)
 \end{align*}
 for $n$ sufficiently large. But $\delta_n^\infty\{B_t(\theta_0)\} = o_p(b_n)$,
 so $\sup_{\theta \in B_{b_n^{-1/2}}(\hat \theta_n^\pi)}  \big|D_n(\theta) \big| = o_p(1)$, as required.
\end{proof}

\subsection*{Proof of Theorem \ref{thm:posterior}}
\begin{proof}
  The normalized exact and approximate posterior distributions are
  $\pi(\theta| y) = L_n(\theta) \pi(\theta)/{Z_n}$
  and
  $\tilde \pi(\theta| y) = \tilde L_n(\theta) \pi(\theta)/\tilde Z_n$,
  where $Z_n = \int L_n(\theta) \pi(\theta) d \theta$
  and $\tilde Z_n = \int \tilde L_n(\theta) \pi(\theta) d \theta$.

  First, we find a Laplace approximation
\[\hat Z_n = (2 \pi)^{-p/2} |J_n^\pi(\hat \theta^\pi_n)|^{-1/2} L_n^\pi(\hat \theta^\pi_n)\]
  to $Z_n$. Then $\log Z_n - \log \hat Z_n = o_p(1)$, because $Z_n$ is a $p$-dimensional
  integral, where $p$ remains fixed as $n \rightarrow \infty$. Similarly,
  $\tilde Z_n$
  may be approximated by using a Laplace approximation. Then
  \begin{align*}
    \log Z_n - \log \tilde Z_n &= \log \hat Z_n - \log \hat{\tilde{Z}}_n + o_p(1) \\
    &= (\log|J_n^\pi(\hat \theta^\pi_n)| - \log |\tilde J_n^\pi(\tilde \theta^\pi_n)|)/2 + \ell_n^\pi(\hat \theta^\pi_n) - \tilde \ell_n^\pi(\tilde \theta^\pi_n) + o_p(1).
  \end{align*}
  Since $\gamma_n^\infty\{B_t(\theta_0)\} = o_p(r_n)$,
  both $r_n^{-1} J_n^\pi(\hat \theta^\pi_n)$ and $r_n^{-1} \tilde J_n^\pi(\hat \theta^\pi_n)$ converge towards $I(\theta_0)$, so
\begin{align*}
  \log|J_n^\pi(\hat \theta^\pi_n)| - \log |\tilde J_n^\pi(\tilde \theta^\pi_n)|
  &= \log|r_n^{-1} J_n^\pi(\hat \theta^\pi_n)| - \log |r_n^{-1} \tilde J_n^\pi(\tilde \theta^\pi_n)| \\
  &= \log|I(\theta_0) + o_p(1)| - \log|I(\theta_0) + o_p(1)| \\
  &= o_p(1).
  \end{align*}
So
  \[\log Z_n - \log \tilde Z_n = \ell_n^\pi(\hat \theta^\pi_n) - \tilde \ell_n^\pi(\tilde \theta^\pi_n) + o_p(1).\]
Since
$\delta_n^\infty\{B_t(\theta_0)\} = o_p(r_n^{-1/2})$ and
  $\gamma_n^\infty\{B_t(\theta_0)\} = o_p(r_n)$,
  we may choose $b_n = o(r_n)$ such that
  $\delta_n^\infty\{B_t(\theta_0)\} = o_p(b_n^{-1/2})$ and
  $\gamma_n^\infty\{B_t(\theta_0)\} = o_p(b_n)$.
 Writing $S_n = B_{b_n^{-1/2}}(\hat \theta_n^\pi)$,
  \begin{align*}
 \sup_{\theta \in S_n}&
 \Big\{ \big| \log \tilde \pi(\theta| y) - \log \pi(\theta | y)
 - (\hat \theta_n^\pi - \theta)^T \nabla_\theta\epsilon_n(\theta) \big| \Big\} \\
 & = \sup_{\theta \in S_n}
 \Big\{\big|
 \tilde \ell_n^\pi(\theta) - \log \tilde Z_n- \ell_n^\pi(\theta)
 + \log Z_n
 - (\hat \theta_n^\pi - \theta)^T \nabla_\theta\epsilon_n(\theta)
 \big| \Big\} \\
 & \leq \sup_{\theta \in S_n}
 \Big\{\big|
 \tilde \ell_n^\pi(\theta) - \tilde \ell_n^\pi(\tilde \theta_n^\pi) - \ell_n^\pi(\theta)
 + \hat \ell_n^\pi(\hat \theta_n^\pi)
 - (\hat \theta_n^\pi - \theta)^T \nabla_\theta\epsilon_n(\theta)
 \big| \Big\} + o_p(1)\\
&= o_p(1),
  \end{align*}
  by Lemma \ref{lemma:LR_diff}. Then
\begin{align*}
  2 d_{TV}\{\tilde \pi(\theta|y), \pi(\theta | y)\}
  &= \int_{\Theta} \big|\tilde \pi(\theta|y) - \pi(\theta|y) \big| d\theta \\
  &= \int_{S_n}  \big|\tilde \pi(\theta|y) - \pi(\theta|y) \big| d\theta + \int_{S_n^C}  \big|\tilde \pi(\theta|y) - \pi(\theta|y) \big| d\theta \\
  &\leq \int_{S_n}  \big|\tilde \pi(\theta|y) - \pi(\theta|y) \big| d\theta + \int_{S_n^C}  2 \big|\pi(\theta|y) \big| d\theta + \int_{S_n^C}  2 \big|\tilde \pi(\theta|y) \big| d\theta \\
  &= \int_{S_n}  \left|\frac{\tilde \pi(\theta|y)}{\pi(\theta|y)} - 1 \right| \pi(\theta|y) d\theta + o_p(1) \\
  &\leq \sup_{\theta \in S_n}  \left|\frac{\tilde \pi(\theta|y)}{\pi(\theta|y)} - 1 \right| + o_p(1) \\
  &= \sup_{\theta \in S_n}  \big|\exp\{\log \tilde \pi(\theta|y) - \log \pi(\theta|y) \} - 1 \big| + o_p(1)\\
  &\leq   \Big|\exp \big\{\sup_{\theta \in S_n} |\log \tilde \pi(\theta|y) - \log \pi(\theta|y)|\big\} - 1 \Big| + o_p(1).
  \end{align*}

But
\begin{align*}
  \sup_{\theta \in S_n} |\log \tilde \pi(\theta|y) - \log \pi(\theta|y)|
  &\leq \sup_{\theta \in S_n} |(\hat \theta^\pi_n - \theta)^T \nabla_\theta \epsilon_n(\theta)| + o_p(1) \\
  &= b_n^{-1/2} \delta_n^\infty\{B_t(\theta_0)\} +  o_p(1)\\
  &= o_p(1).
  \end{align*}
So $d_{TV}\{\tilde \pi(\theta|y), \pi(\theta | y)\} = o_p(1)$,
as claimed.
\end{proof}

\bibliography{ogden}
\bibliographystyle{plainnat}

\renewcommand{\thesection}{\arabic{section}}

\setcounter{equation}{0}
\setcounter{figure}{0}
\setcounter{table}{0}
\setcounter{page}{1}
\makeatletter
\renewcommand{\theequation}{S\arabic{equation}}
\renewcommand{\thefigure}{S\arabic{figure}}
\renewcommand{\bibnumfmt}[1]{[S#1]}
\renewcommand{\citenumfont}[1]{S#1}

\title{Supplementary material for \\ On asymptotic validity of naive inference with an approximate likelihood}
\maketitle

\section{Finding $\delta_{(i)}(\theta)$ in Example 3.1}
Recall $Y_i$ is the number of successes out of $m = m_n$
trials on item $i$. We study how
$\delta_{(i)}(\theta) = \norm{(d/{d \theta}) \epsilon_{(i)}(\theta)}$
varies with $m$.

Write
\[f(b; y_i) =  -y_i \log \left\{\text{logit}^{-1}(b)\right\} + (m - y_i)  \log\left\{1 -  \text{logit}^{-1}(b)\right\}\]
and
\[g(b; \theta, y_i) = f(b; y_i) - \log \phi(b; 0, \theta),\]
so that
\[L_{(i)}(\theta) = \int_{-\infty}^{\infty} \exp\{ - g(b; \theta, y_i)\} db. \]
In the following, we drop the data $y_i$ from the notation
for convenience.
Write $\hat b(\theta)$ for the maximizer of $g(.,\theta)$, and
\[\hat g_r(\theta) = \frac{\partial^r g}{\partial b^k} (\hat b(\theta); \theta).\]

By equation (4) of \cite{Shun1995}, the error in the Laplace approximation
to the log-likelihood $\ell_{(i)}(\theta)$ is
\begin{equation}
\epsilon_{(i)}(\theta) = \sum_{l=1}^\infty \frac{1}{2 l!} \sum_{P \in \mathcal{P}_{2l}} n_{2}(P) (-1)^v 
\hat g_{|p_1|}(\theta) \ldots \hat g_{|p_v|}(\theta) \{\hat g_2(\theta)\}^{-l},
\label{eqn:epsilon}
\end{equation}
where $P = p_1 | \ldots | p_v$ is a partition of $2 l$ indices into $v$ blocks
of size $3$ or more, and $n_2(P)$ is the number of partitions $Q$ of $2 l$ indices
into $l$ blocks of size $2$, such that $Q$ is
 complementary to $P$.

Write
$h_P(\theta) = \hat g_{|p_1|}(\theta) \ldots \hat g_{|p_v|}(\theta) \{\hat g_2(\theta)\}^{-l}.$
Then $h_P(\theta) = O_p(m^{v - l})$, since $\hat g_r(\theta) = O_p(m)$ for each $r$.

Differentiating \eqref{eqn:epsilon} gives
\begin{equation}
\frac{d}{d \theta} \epsilon_{(i)}(\theta) = \sum_{l=1}^\infty \frac{1}{2 l!} \sum_{P \in \mathcal{P}_{2l}} n_{2}(P) (-1)^v \frac{d}{d \theta} h_P(\theta),
\label{eqn:size_of_delta}
\end{equation}
and
\begin{equation}
\frac{d}{d \theta} h_P(\theta) = \sum_{i=1}^v \Big[\hat g^\prime_{|p_i|}(\theta) \prod_{j \not = i} \hat g_{|p_j|}(\theta) \{\hat g_2(\theta) \}^{-l} - \prod_{j=1}^v l \hat g_{|p_j|}(\theta)  \hat g_2^\prime(\theta) \{\hat g_2(\theta) \}^{-(l+1)}\Big],
\label{eqn:size_of_h_P_deriv}
\end{equation}
where $\hat g_r^\prime(\theta) = (d / d \theta) \hat g_r(\theta)$.

For each $r$, we have
\begin{align}
\hat g_r^\prime(\theta) &= \frac{d}{d \theta} \left\{g^{(r)}(\hat b(\theta); \theta) \right\} \nonumber \\
&= \frac{\partial g^{(r)}}{\partial \theta} (\hat b(\theta); \theta ) + \frac{d \hat b(\theta)}{d \theta} \hat g_{r+1}(\theta).
\label{eqn:size_of_g_r_deriv}
\end{align}
We now study the size of each of the terms in \eqref{eqn:size_of_g_r_deriv}.
We have
\begin{equation}
  \frac{\partial g^{(r)}}{\partial \theta} (b; \theta) = \frac{\partial}{\partial \theta} \left\{ - \frac{\partial^r}{\partial b^r} \log \phi(b; 0, \theta)\right\} = O_p(1),
  \label{eqn:size_of_g_r_fixed_b_deriv}
\end{equation}
and
\begin{equation}
  \hat g_{r+1}(\theta) = O_p(m).
  \label{eqn:size_of_g_rp1}
\end{equation}
For each $\theta,$
$\hat b(\theta)$ satisfies $g_1(\hat b(\theta); \theta) = 0$.
Differentiating this with respect to $\theta$,
\[\frac{d \hat b(\theta)}{d \theta} g_2(\hat b(\theta); \theta) + \frac{\partial g_1}{\partial \theta}(\hat b(\theta); \theta) = 0.\]
But
\[\frac{\partial g_1}{\partial \theta}(b; \theta) = -2 b \theta^{-3},\]
so 
\begin{equation}
\frac{d \hat b(\theta)}{d \theta} = 2 \hat b(\theta) \theta^{-3} \{\hat g_2(\theta) \}^{-1} = O_p(m^{-1}).
\label{eqn:size_of_b_deriv}
\end{equation}

Substituting \eqref{eqn:size_of_g_r_fixed_b_deriv},
\eqref{eqn:size_of_g_rp1} and \eqref{eqn:size_of_b_deriv} 
 into \eqref{eqn:size_of_g_r_deriv} gives that
$ \hat g_r^\prime(\theta) = O_p(1)$ for each $r$.
From \eqref{eqn:size_of_h_P_deriv} we then
have
\[\frac{d}{d \theta} h_P(\theta) = O_p(m^{v - l  - 1}).\]
The highest order terms in \eqref{eqn:size_of_delta} come from 
partitions with $(l, v) = (2, 1)$ or $(3, 2)$, and so
 $\delta_{(i)}(\theta) = O_p(m^{-2})$.

\section{Finding $\delta_m^{(k)}(\beta)$ in Example 3.2}

\cite{Kaufman1949} provides an exact expression
for the normalizing constant for an Ising model
on an $n \times m$ lattice, with $\alpha = 0$ and periodic boundary, as
\[Z_{n \times m}(0, \beta) = \left\{2 \sinh(2 \beta)\right\}^{nm/2} \bar A_{n,m}(\beta) / 2,\]
where 
\[\bar A_{n,m}(\beta) = A^{(1)}_{n,m}(\beta) + A^{(2)}_{n,m}(\beta)
+ A^{(3)}_{n,m}(\beta) + A^{(4)}_{n,m}(\beta),\]
and
\begin{align*}
A^{(1)}_{n,m}(\beta) &= \prod_{q=0}^{n} 2 \cosh \left(m  \, a_{2q+1, n}(\beta)/2\right), 
&A^{(2)}_{n,m}(\beta) = \prod_{q=0}^{n} 2 \sinh \left(m  \, a_{2q+1, n}(\beta)/2\right), \\
A^{(3)}_{n,m}(\beta) &= \prod_{q=0}^{n} 2 \cosh \left(m  \, a_{2q, n}(\beta)/2\right), 
&A^{(4)}_{n,m}(\beta) = \prod_{q=0}^{n} 2 \sinh \left(m  \, a_{2q, n}(\beta)/2\right)
\end{align*}
where
\[a_{l,n}(\beta) = \cosh^{-1} \left\{\cosh(2 \beta)^2/\sinh(2 \beta) - \cos(\pi l /n) \right\}\]
for $l \geq 1$, and
$a_{0,n}(\beta) = a_0(\beta) = 2 \beta + \log \left\{\tanh(\beta)\right\}$. 
 
Using the approximation $Z^{(k)}_{m \times m}(\beta)$ to
$Z_{m \times m}(\beta)$, the error in the log-likelihood is
\[\epsilon^{(k)}_{m}(\beta) = (m-k+1)\log \bar A_{k,m}(\beta)
-(m-k)\log \bar A_{k-1,m}(\beta) - \log \bar A_{m,m}(\beta).\]

Differentiating this with respect to $\beta$,
\begin{equation}
  \label{eqn:error_in_score}
\frac{d}{d\beta} \epsilon^{(k)}_{m}(\beta) = 
 (m-k+1) \frac{d}{d\beta}\left\{\log \bar A_{k,m}(\beta)\right\}
-(m-k)\frac{d}{d\beta}\left\{\log \bar A_{k-1,m}(\beta)\right\} - \frac{d}{d\beta}\left\{\log \bar A_{m,m}(\beta)\right\},
\end{equation}
and
\[ \frac{d}{d\beta}\left\{\log \bar A_{n,m}(\beta)\right\}
= \sum_{i=1}^4  \frac{d}{d\beta} \left\{\log A_{n,m}^{(i)}(\beta)\right\} r_{n,m}^{(i)}(\beta),\]
where
$r_{n,m}^{(i)}(\beta) = A_{n,m}^{(i)}(\beta) / \bar A_{n,m}(\beta)$.

We have
\begin{align*}
\frac{d}{d\beta} \left\{\log A_{n,m}^{(1)}(\beta)\right\}
&= m/2 \sum_{q=0}^n a^\prime_{2q+1, n}(\beta) \tanh(m a_{2q+1, n}(\beta)/2) \\
&= m/2  \sum_{q=0}^n a^\prime_{2q+1, n}(\beta) + O(m \exp\{-a_0(\beta) m\})
\end{align*}
as $m \rightarrow \infty$, since
$\tanh(x) = 1 + O(\exp\{-2 x\})$ as $x \rightarrow \infty$,
and
$a_{2q + 1, n}(\beta) \geq a_0(\beta) > 0$.

Similar expressions may be obtained
for the other $\frac{d}{d\beta} \left\{\log A_{n,m}^{(i)}(\beta)\right\}$,
and combining these gives
\[ \frac{d}{d\beta}\left\{\log \bar A_{n,m}(\beta)\right\}
= m S^{(o)}_n(\beta) r_{n,m}^{(o)}(\beta) + m S^{(e)}_n(\beta) r_{n,m}^{(e)}(\beta) + O(m \exp\{-a_0(\beta) m\})\]
where
$S^{(o)}_n = \sum_{q=0}^n a^\prime_{2q+1, n}(\beta),$
$S^{(e)}_n = \sum_{q=0}^n a^\prime_{2q, n}(\beta),$
$r_{n,m}^{(o)}(\beta) = r_{n,m}^{(1)}(\beta) + r_{n,m}^{(2)}(\beta)$
and 
$r_{n,m}^{(e)}(\beta) = r_{n,m}^{(3)}(\beta) + r_{n,m}^{(4)}(\beta)$.
Define
\[f(x; \beta) = d_\beta \,
\{-1 + c_\beta - \cos(x)\}^{-1/2} \{1 + c_\beta - \cos(x)\}^{-1/2}\]
where
$d_\beta = 4 \cosh(2 \beta) - 2 \cosh(2 \beta) \coth(2 \beta)^2$
and
$c_\beta = \cosh(2 \beta)^2 / \sinh(2 \beta)\null.$
Then $a^\prime_{j,n}(\beta) = f(j \pi / n; \beta),$
and
$n^{-1} S_n^{(o)}(\beta)$
and $n^{-1} S_n^{(e)}(\beta)$
are both trapezium rule
approximations
to $I(\beta) = \frac{1}{2 \pi}\int_{0}^{2\pi} f(x; \beta) dx$.
Write $R_n^{(o)}(\beta) = n^{-1} S_n^{(o)}(\beta) - I(\beta)$
and
$R_n^{(e)}(\beta) = n^{-1} S_n^{(e)}(\beta) - I(\beta)$
for the error in each of these approximations
to the integral.

\begin{lemma}
  \label{lemma:size_of_remainder}
  For each $\beta < \beta_c$,
  $R_n(\beta) = \max\{|R_n^{(e)}(\beta)|, |R_n^{(o)}(\beta)|\}
  = O(\exp\{-b_\beta n\})$,
  where
  $b_\beta = 2 \cosh^{-1}\{-1 + \cosh(2 \beta)^2 / \sinh(2 \beta)\}.$
\end{lemma}
\begin{proof}
  We apply the results of \cite{Trefethen2014}
  to show exponentially fast convergence of these
  trapezium rule approximations to $I(\beta)$.
  These results depend on properties of 
the integrand $f(z,\beta)$, considered
 as a function of complex-valued $z$.
  There are a branch points of $f(z, \beta)$ at a
  distance  
  $a_\beta = \cosh^{-1}\{-1 + \cosh(2 \beta)^2 / \sinh(2 \beta)\}$
  from the real axis, and the function is analytic for 
  $-a_\beta < \operatorname{Im} z < a_\beta$, so
  by Theorem 3.2 of \cite{Trefethen2014},
  $|R_n^{(o)}(\beta)| = O(\exp\{-2 a_\beta n\}) = O(\exp\{-b_\beta n\}).$
  
  The same argument holds with $R_n^{(e)}(\beta)$ in place
  of $R_n^{(o)}(\beta)$, so $R_n(\beta) =  O(\exp\{-b_\beta n\}),$
    as required.
\end{proof}

We now prove the main result.

\begin{lemma}
  \label{lemma:error_in_score}
  If $k \rightarrow \infty$ as $m \rightarrow \infty$, 
  $\delta^{(k)}_{m}(\beta) = O(m^2 k \exp\{-b_\beta k\}) + o(1).$
\end{lemma}

\begin{proof}
We have
\[\frac{d}{d\beta}\left\{\log \bar A_{n,m}(\beta)\right\}
= m n I(\beta) + m n t_{n, m}(\beta)  + O(m \exp\{-a_0(\beta) m\})\]
where
$t_{n,m}(\beta) = R^{(o)}_n(\beta) r_{n,m}^{(o)}(\beta) +  R^{(e)}_n(\beta) r_{n,m}^{(e)}(\beta)$. 

Substituting this into \eqref{eqn:error_in_score},
the contributions from the $m n I(\beta)$ terms cancel, and
the combined remainder terms are always $o(1)$, since
$m^2 \exp\{-a_0(\beta) m\} = o(1)$. We are left with
\[\frac{d}{d\beta} \epsilon^{(k)}_{m}(\beta) = 
 (m-k+1) m t_{k,m}(\beta)
-(m-k) m t_{k-1, m}(\beta) - m t_{m,m}(\beta) + o(1). \]
Then
\begin{align*}
  \delta^{(k)}_{m}(\beta) &= \Big|\frac{d}{d\beta} \epsilon^{(k)}_{m}(\beta) \Big| \\
  &\leq (m-k+1) m k |t_{k, m}(\beta)| + (m - k)  m (k-1) |t_{k-1, m}(\beta)| + m^2 |t_{m, m}(\beta)| + o(1) \\
  &\leq (m - k + 1) m k R_k(\beta) + (m - k) m (k-1) R_{k-1}(\beta) + m^2 R_m(\beta) + o(1)
\intertext{since $|t_{n, m}(\beta)| \leq  |R^{(o)}_n(\beta)| r_{n,m}^{(o)}(\beta) + |R^{(e)}_n(\beta)| r_{n,m}^{(e)}(\beta) \leq R_n(\beta)$}
  &= O(m^2 k \exp\{-b_\beta k\}) + o(1)
\end{align*}
by Lemma \ref{lemma:size_of_remainder}, as required.
\end{proof}

\bibliography{../ogden}
\bibliographystyle{plainnat}

\end{document}